\documentclass[11pt,a4paper]{article}

\usepackage[T1]{fontenc}
\usepackage[utf8]{inputenc}
\usepackage{lmodern}
\usepackage[a4paper,margin=1in]{geometry}
\usepackage{amsmath,amssymb,amsthm,mathtools}
\usepackage{graphicx}
\usepackage{xcolor}
\usepackage{booktabs}
\usepackage[numbers,sort&compress]{natbib}
\usepackage{siunitx}
\usepackage{subcaption}
\usepackage{bm}
\usepackage{accents}
\usepackage{calc,fancyvrb}
\usepackage{algorithm,algpseudocode}
\usepackage{algorithmicx}
\usepackage{placeins}
\usepackage{microtype}
\usepackage{hyperref}
\usepackage{cleveref}

\makeatletter
\providecommand*{\theHALG@line}{}
\renewcommand*{\theHALG@line}{\thealgorithm.\arabic{ALG@line}}
\makeatother

\newcommand{\rev}[1]{#1}
\newcommand{\edit}[1]{#1}
\newenvironment{revisionblock}{\par\begingroup}{\par\endgroup}

\algrenewcommand\alglinenumber[1]{\footnotesize #1:}

\graphicspath{{figs/}}

\newcommand{\R}{\mathbb{R}}
\newcommand{\Z}{\mathbb{Z}}
\newcommand{\grad}{\nabla}
\newcommand{\lap}{\Delta}
\newcommand{\free}{\Omega_{\mathrm{free}}}
\newcommand{\obs}{\mathcal{O}}
\newcommand{\xx}{\mathbf{x}}

\newcommand{\norm}[1]{\left\lVert #1 \right\rVert}

\newcommand{\Gobs}{\Gamma_{\mathrm{obs}}}
\newcommand{\Gout}{\Gamma_{\mathrm{out}}}
\newcommand{\Ggoal}{\Gamma_{\mathrm{goal}}}

\newtheorem{proposition}{Proposition}[section]
\newtheorem{lemma}{Lemma}[section]
\newtheorem{definition}{Definition}[section]

\newtheorem{remark}{Remark}[section]

\numberwithin{equation}{section}

\renewcommand{\theHequation}{\arabic{section}.\arabic{equation}}

\begin{document}
\renewcommand{\floatpagefraction}{.85}
\renewcommand{\textfraction}{.05}

\title{An unfitted boundary algebraic equation method with static-dynamic reduction for evolving implicit geometries}

\author{\parbox{0.94\textwidth}{\centering
Yizhen Huang\textsuperscript{1} \quad
Qing Xia\textsuperscript{1,2}\thanks{Corresponding author:
\href{mailto:qxia@kean.edu}{qxia@kean.edu}}\\[0.6em]
\small \textsuperscript{1}College of Science, Mathematics and Technology,
Wenzhou-Kean University,\\
\small Wenzhou 325060, Zhejiang, China\\
\small \textsuperscript{2}International Frontier Interdisciplinary Research
Institute, Wenzhou-Kean University,\\
\small Wenzhou 325060, Zhejiang, China
}}
\date{}

\maketitle

\begin{abstract}
Repeated elliptic solves on domains with evolving boundaries arise in
moving-interface simulation, design, and reactive navigation. Even when a
fixed Cartesian grid avoids remeshing, rebuilding all boundary interactions for
every configuration can limit the efficiency of repeated solves. We develop a
static--dynamic boundary reduction for an unfitted lattice Green's function
method on prescribed moving planar domains. Like boundary integral and boundary
element methods, the formulation reduces the problem to boundary-supported
unknowns through a Green representation. Its construction, however, reverses
the usual order: the Cartesian operator is discretized before the Green
representation is formed, rather than representing the continuous problem
first and then discretizing the boundary. This discretize-then-represent
viewpoint avoids boundary meshes and singular quadrature. The method also
separates interactions associated with stationary geometry from those affected
by motion, reuses the invariant part throughout a simulation, and updates only
couplings involving the changing boundary. Boundary conditions are imposed at
true interface intersections, lattice-kernel data are reused, and the interior
field is reconstructed by a fast sine-transform solver. The principal
contribution is an implemented and validated update strategy for translating,
deforming, appearing, and topology-changing obstacles.
\end{abstract}

\noindent\textbf{Keywords:} difference potentials; lattice Green's functions;
unfitted methods; moving-boundary Laplace problems; harmonic navigation; fast
Poisson solvers
\medskip

\section{Introduction}
\label{sec:intro}

Solving an elliptic boundary value problem on a domain whose boundary is
complicated, known only implicitly, and---most consequentially---\emph{changing
in time} is a recurring computational task in numerical analysis \cite{osher1988fronts,lehrenfeld2019eulerian}. Free-boundary and
moving-interface problems, fluid--structure interaction, and reactive
navigation all require an elliptic solve to be repeated as the geometry evolves \cite{bazilevs2013computational,connolly1990path,kim1992real}.
Body-fitted methods can be effective for moderate boundary motion, for example
through mesh deformation or arbitrary Lagrangian--Eulerian techniques \cite{donea2004arbitrary,tezduyar2001finite}, but large
deformation, topology change, or newly appearing obstacles may require remeshing
and repeated operator assembly \cite{osher1988fronts,shamanskiy2021mesh}. In that regime the geometry update can become a
substantial part of the total cost.

\edit{This paper develops an unfitted method for a sequence of steady Laplace
problems on prescribed moving implicit domains. It does not solve a coupled
free-boundary evolution problem: the geometry at each frame is supplied to the
elliptic solver.} The driving
application---which motivates the boundary conditions, the accuracy metric, and
the dynamic experiments---is \emph{dynamic harmonic path planning}. Given a free
region $\free$ bounded by obstacles, one solves Laplace's equation with the goal
held at a low Dirichlet value and the obstacle and outer boundaries held at a
high Dirichlet value; the resulting potential $u$ is harmonic, and by the
maximum principle it has no interior local minima, so the descent flow of
$-\grad u$ reaches the goal from almost every starting
point~\cite{connolly1990path,connolly1993applications,kim1992real,
rimon1990exact}. This local-minimum-free guarantee is the
defining advantage of harmonic fields over the heuristic superposition of
attractive and repulsive potentials~\cite{khatib1986real}, which can stall at
spurious minima. When obstacles move, the field must be \emph{re-solved}, and
the efficiency of that re-solve is exactly the moving-geometry question above.

\edit{The LGF-BAE discretization separates reusable grid, kernel, and
reconstruction data from boundary-local data changed by motion. The directly
assembled density matrix admits a corresponding block partition;
\Cref{sec:method-dynamic} states precisely which block is invariant.}

We solve the harmonic boundary value problem on a fixed Cartesian grid with an
unfitted boundary algebraic equation (BAE) method built on lattice Green's
function (LGF) and difference-potentials ideas
\cite{xia2025geom,ryaben2012method,petropavlovsky2024unsteady}.
\edit{Like boundary integral and boundary element methods, LGF-BAE replaces a
volumetric system by boundary-supported unknowns connected through a Green
representation. The essential difference is the order of construction.
Classical boundary integral and boundary element formulations first represent
the continuous solution and then discretize the resulting boundary equation
\cite{helsingfast2011,bremer2010efficient}; they are
represent-then-discretize methods. LGF-BAE first discretizes the elliptic
operator on the Cartesian lattice and then uses the lattice Green's function
and a discrete Green identity to obtain the boundary representation; it is a
discretize-then-represent method. This ordering removes the need for a
body-fitted boundary mesh and singular boundary quadrature.} The resulting
solution is a discrete layer potential supported on boundary-adjacent exterior
grid points. Compactly supported local basis functions enforce Dirichlet data
at the true boundary--grid intersections.

This construction belongs to the broader family of fixed-background-grid
methods, including immersed interface method~\cite{leveque1994immersed},
immersed boundary method~\cite{peskin2002immersed}, kernel-free boundary integral 
method~\cite{ying2013kernel,ying2007kernel,zhou2024correction} and CutFEM
\cite{burman2015cutfem} discretizations; level-set methodology is reviewed in
\cite{gibou2018review}. CutFEM provides a flexible variational framework and
handles complex interfaces, but moving cut cells generally require updated
interface quadrature and stabilization data. In the present LGF-BAE setting,
motion changes boundary intersections, local interpolation weights, and the
coupled boundary blocks, while the Cartesian operator, tabulated Green's
function, and sine-transform reconstruction remain unchanged. This narrower elliptic
setting therefore exposes unusually direct reuse under geometry motion, at the
cost of dense boundary interactions unless local assembly or compression is
used. \edit{Closely related Cartesian-grid boundary-integral methods use fast
box solvers to evaluate boundary operators for moving-interface problems and
couple those solves to the interface evolution~\cite{zhou2026cartesian}; an
earlier level-set/boundary-integral method likewise accommodates moving
boundaries and topology change~\cite{garzon2005coupled}. Our narrower setting
uses an exact lattice kernel and table lookup for repeated prescribed-geometry
Dirichlet solves, rather than evolving the interface from the PDE solution.}

The harmonic-navigation model is classical and is used here as a demanding
application of a repeated Laplace solve. The underlying difference-potentials
discretization also builds on earlier work~\cite{xia2025geom,ryaben2012method,petropavlovsky2024unsteady,albright2015high}.
The contributions of the present paper are the following.
\begin{enumerate}
\item \textbf{Moving-geometry algebraic organization.} We separate the
boundary algebraic data into stationary and geometry-dependent components and
\edit{identify which blocks of the directly assembled density matrix change
when an implicit obstacle moves.}
\item \textbf{Static-block Schur complement.} \edit{For the density matrix,
we derive a Schur reduction whose static block can be factored once when the
static boundary neighborhood and ordering remain fixed. We also show why the
global inverse in the trace formulation invalidates an unqualified static-block
reuse claim. The implemented update caches the fixed collocation stencils and
$B_{\mathrm{ss}}$ factorization and assembles only the dynamic-coupled blocks on
later frames.}
\item \textbf{Gradient-sensitive verification.} \edit{We verify potential,
gradient-magnitude, and gradient-angle errors over the full interior, a fixed
physical bulk, and the first two interior grid layers, because path extraction
is controlled by $-\grad_h u/\norm{\grad_h u}$ rather than by the potential
alone.}
\item \textbf{Evolving-geometry validation.} \edit{The cached update is tested
under translation and a merge--split--merge topology change, including a case
with up to 23\% dynamic unknowns. Additional demonstrations cover moving,
deforming, appearing, and topology-changing obstacles without remeshing.}
\end{enumerate}

The remainder of the paper is organized as follows. \Cref{sec:problem} states
the harmonic boundary value problem and explains why gradient accuracy is the
right verification metric. \Cref{sec:method} presents the fixed-geometry
LGF-BAE discretization. \Cref{sec:method-dynamic} develops the moving-geometry
static/dynamic decomposition and presents the full-trace and cached block-update
algorithms separately.
\Cref{sec:analysis} discusses solvability, consistency, conditioning,
complexity and re-solve cost, and establishes the algebraic equivalence of the
block reduction and the full density solve (\Cref{sec:schur-equivalence}).
\Cref{sec:self-convergence} presents the convergence, same-grid equivalence,
and performance experiments. \Cref{sec:application} gives the dynamic
planning experiments, including the path-extraction procedure used there.
\Cref{sec:discussion} compares the method with fast marching and
sampling-based planning, discusses limitations, and outlines extensions.
\Cref{sec:conclusion} concludes.

\section{Model problem: the harmonic navigation field}
\label{sec:problem}

\subsection{Free space and the boundary value problem}
\label{sec:freespace}

Let $D \subset \R^2$ be a bounded outer domain, let
$\obs_1(t),\dots,\obs_m(t)\subset D$ be closed obstacle regions, and let
$G\subset D$ be a fixed closed goal region. The obstacles are described
implicitly by level-set functions $\phi_k(\xx,t)$~\cite{osher1988fronts}, with
$\obs_k(t)=\{\phi_k(\xx,t)\le0\}$. See Figure~\ref{fig:domain} for an illustration. The goal is also treated as an interior
boundary, $\Ggoal=\partial G$, rather than as a point constraint. At each time
$t$ the computational free space is
\begin{equation}\label{eq:free-space}
\free(t) = D\setminus\Bigl(\overline{G}\cup\bigcup_{k=1}^{m}\obs_k(t)\Bigr),
\qquad
\partial\free(t)=\Ggoal\cup\Gobs(t)\cup\Gout,
\end{equation}
where $\Gobs(t)=\bigcup_k\partial\obs_k(t)$ and $\Gout=\partial D$. The
obstacles need not be convex or simply connected and may be added, removed,
translated, or deformed by updating their level-set functions. A composite
level set, for instance
\begin{equation}
\Psi(\xx,t)=\max\bigl(\psi_{\rm out}(\xx),-\phi_G(\xx),
-\phi_1(\xx,t),\dots,-\phi_m(\xx,t)\bigr),
\end{equation}
encodes the instantaneous free space as $\{\Psi\le0\}$. This implicit
representation is what allows the geometry to move without remeshing.

\begin{figure}[htbp]
\centering
\includegraphics[width=0.35\linewidth]{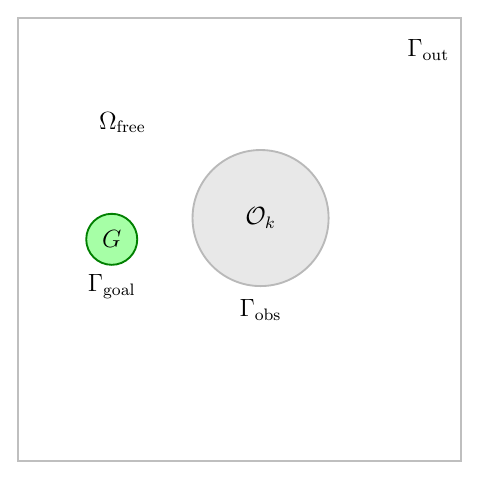}
\caption{\edit{Computational geometry: a fixed outer boundary encloses the free
space, moving obstacles, and the interior goal boundary. The harmonic field is
zero on the goal and one on the obstacle and outer boundaries.}}
\label{fig:domain}
\end{figure}

We seek a harmonic navigation field $u(\cdot,t):\free(t)\to\R$ satisfying
\begin{subequations}\label{eq:bvp}
\begin{align}
\lap u &= 0 && \text{in } \free(t), \\
u &= 0 && \text{on } \Ggoal, \\
u &= 1 && \text{on } \Gobs(t)\cup\Gout .
\end{align}
\end{subequations}
 \renewcommand{\theHequation}{\arabic{section}.\arabic{equation}}
 For a fixed time $t$, the strong maximum principle gives
$0<u<1$ in $\free(t)$, with no interior local extrema. In the dynamic setting the
same boundary value problem is solved repeatedly on the instantaneous domain
$\free(t)$. In the fixed-time discretization and analysis below, we suppress the
time variable and write $\free$ for the instantaneous computational domain.

\subsection{Why gradient accuracy is the relevant metric}
\label{sec:gradient-metric}

The planned trajectory is the integral curve of the normalized descent field,
\begin{equation}\label{eq:descent-ode}
\dot{\xx}(t) = -\,\frac{\grad u(\xx(t))}{\norm{\grad u(\xx(t))}},
\qquad \xx(0)=\xx_s,
\end{equation}
terminating on entry to $G$. The following classical fact motivates the analysis
target of \Cref{sec:analysis,sec:self-convergence}.

We assume $\free\subset\R^2$ is bounded, open, and connected, with
$\partial\free=\Ggoal\cup\Gobs\cup\Gout$ and $\Ggoal=\partial G$ of class $C^2$
(an interior-ball condition at $\Ggoal$ suffices). Let
$u\in C(\overline{\free})\cap C^\infty(\free)$ be the harmonic solution of
\eqref{eq:bvp}; near $\Ggoal$ we have $u\in C^1$. Since the Dirichlet data take
the two distinct values $0$ and $1$, $u$ is nonconstant.

The continuous harmonic navigation field has no interior local minima by
the maximum principle. Its isolated interior critical points are saddles, and
only a measure-zero set of initial conditions terminates at such saddles. Hence
the continuous model is locally-minimum-free for almost every starting point. In
the discrete setting, however, the computed trajectory is controlled by the
reconstructed field $\grad_h u$. This makes gradient accuracy, especially near
unfitted boundaries, a central numerical requirement. \edit{These are
continuous statements; they do not provide an arrival or collision-avoidance
guarantee for a finite-step NAG trajectory based on an interpolated
discrete gradient.}



\section{The unfitted lattice Green's function method}
\label{sec:method}

We first summarize the fixed-geometry LGF-BAE discretization of  the local-basis
difference-potentials method~\cite{xia2025geom}. The moving-geometry
decomposition,  the methodological core of this paper, is developed separately in
\Cref{sec:method-dynamic}.

\subsection{Grid and boundary neighborhood}
\label{sec:grid}
The free space is embedded in the Cartesian grid $h\Z^2$. Grid points are
classified as interior $M^+ = h\Z^2\cap\overline{\free}$ and exterior
$M^- = h\Z^2\setminus M^+$. With the five-point stencil $\mathcal N^5_{jk}$,
set $N^\pm = \{\mathcal N^5_{jk} : (x_j,y_k)\in M^\pm\}$. The discrete boundary
neighborhood is the stencil overlap $\gamma = N^+\cap N^-$, split as
$\gamma^+ = \gamma\cap M^+$ and $\gamma^- = \gamma\setminus\gamma^+$. Layer
densities live on $\gamma^-$; the induced potential is evaluated on $\gamma^+$
and throughout $M^+$. This classification (in addition to boundary intersection points) is the only geometry-dependent data;
the grid is fixed. \edit{The square and triangular markers in
\Cref{fig:gamma} show the two sides of this discrete boundary neighborhood.}

\begin{figure}[htbp]
\centering
\includegraphics[width=0.35\linewidth]{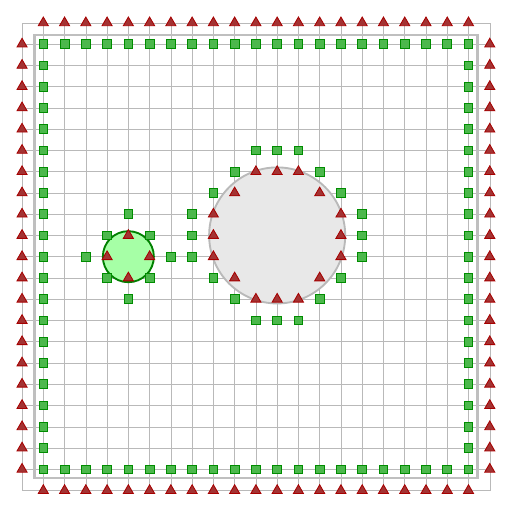}
\caption{Classification of Cartesian grid points into $\gamma^+$ (squares) and $\gamma^-$ (triangles) near an unfitted boundary.}
\label{fig:gamma}
\end{figure}

\subsection{Lattice Green's function}
\label{sec:lgf}
The lattice Green's function (LGF) is the fundamental solution of the discrete
Laplacian on the infinite grid, and it plays the role for the discrete operator
that the free-space kernel $-(2\pi)^{-1}\log r$ plays in classical potential
theory. With the (unscaled) five-point operator
\begin{equation}\label{eq:five-point}
[Au](m) = 4u(m) - u(m+e_1)-u(m-e_1)-u(m+e_2)-u(m-e_2),
\qquad m\in\Z^2,
\end{equation}
$e_1=(1,0)$, $e_2=(0,1)$, the LGF $G$ is defined by
\begin{equation}\label{eq:lgf-def}
A G = \delta_0, \qquad
\delta_0(m) = \begin{cases} 1, & m=(0,0),\\ 0, & m\neq(0,0).\end{cases}
\end{equation}
Diagonalizing $A$ by the Fourier transform on $\Z^2$ gives the formal
representation
\begin{equation}\label{eq:lgf-fourier-raw}
\widetilde G(m) = \frac{1}{(2\pi)^2}\int_{-\pi}^{\pi}\!\!\int_{-\pi}^{\pi}
\frac{\cos(\xi_1 m_1 + \xi_2 m_2)}{4-2\cos\xi_1-2\cos\xi_2}\,d\xi_1\,d\xi_2 ,
\end{equation}
whose symbol $4-2\cos\xi_1-2\cos\xi_2$ vanishes at $\xi=0$. The resulting zero
Fourier mode makes \eqref{eq:lgf-fourier-raw} divergent---a reflection of the fact
that $A$ annihilates constants and therefore has no finite-energy inverse on the
whole lattice. The method instead uses the \emph{normalized} kernel
\begin{equation}\label{eq:lgf-fourier}
G(m) = \frac{1}{(2\pi)^2}\int_{-\pi}^{\pi}\!\!\int_{-\pi}^{\pi}
\frac{\cos(\xi_1 m_1 + \xi_2 m_2)-1}{4-2\cos\xi_1-2\cos\xi_2}\,d\xi_1\,d\xi_2 ,
\end{equation}
in which subtracting $1$ in the numerator cancels the constant mode and renders
the integral convergent~\cite{gillman2014fast,liska2014parallel}. The additive constant is fixed by the normalization $G(0)=0$, and this same
normalized kernel is used consistently in all layer-potential and boundary
operators below.

The normalized kernel takes the exact values
\begin{equation}\label{eq:lgf-values}
G(0,0)=0, \qquad
G(\pm1,0)=G(0,\pm1)=-\tfrac14, \qquad
G(\pm1,\pm1)=-\tfrac{1}{\pi},
\end{equation}
and is evaluated in two regimes. For small $|m|$ the double integral
\cref{eq:lgf-fourier} is reduced to a stable one-dimensional integral and
evaluated once by quadrature. For large $|m|$ it is replaced by the asymptotic
expansion~\cite{liska2014parallel}
\begin{equation}\label{eq:lgf-asymp}
\begin{aligned}
G(m) ={}& -\frac{1}{2\pi}\Bigl(\log|m| + \gamma_E + \tfrac{\log 8}{2}\Bigr)
+ \frac{\cos(4\theta)}{24\pi|m|^2}
+ \frac{25\cos(8\theta)+18\cos(4\theta)}{480\pi|m|^4} \\
&+ \frac{490\cos(12\theta)+459\cos(8\theta)}{2016\pi|m|^6}
+ O(|m|^{-8}),
\qquad \theta=\operatorname{atan2}(m_2,m_1),
\end{aligned}
\end{equation}
whose leading term recovers the continuous Green's function
$-(2\pi)^{-1}\log|m|$ up to the lattice constant $\gamma_E+\tfrac12\log 8$
($\gamma_E$ the Euler--Mascheroni constant), and whose angular corrections encode
the $O(h^2)$ anisotropy of the five-point stencil. \edit{In the implementation,
quadrature is used for $|m|\le 30$ and the displayed asymptotic formula for
$|m|>30$. The switch is an implementation choice; no claim of
full-double-precision agreement at the crossover is made without a separate
kernel-error study.} A single tabulation is reused by every geometry.

Two consequences of \eqref{eq:lgf-values} and \eqref{eq:lgf-asymp} are central to the method.
First, unlike the continuous kernel, $G$ is \emph{finite at the origin}: the
discrete layer potentials built from $G$ never produce the singular or
near-singular integrands of classical boundary integral equations, so all kernel
interactions are exact finite table values and no special-purpose singular
quadrature is required. Second, $G$ depends only on the offset $m-n$ between two
grid points, so its values are tabulated \emph{once} for every offset vector that
can occur within the auxiliary computational box. As the geometry moves, new
boundary pairs reuse existing table entries rather than triggering fresh kernel
evaluations---the property that makes the per-frame re-solve of
\Cref{sec:method-dynamic} cheap.

\subsection{Layer potentials and the boundary algebraic relation}
The solution in $M^+$ is a discrete single-layer potential supported on
$\gamma^-$,
\begin{equation}\label{eq:single-layer}
u(m) = \sum_{n\in\gamma^-} G(m-n)\, q(n),\qquad m\in M^+,
\end{equation}
which satisfies $Au=0$ throughout $M^+$ exactly because the sources lie outside.
Restricting to the boundary sets gives $u_{\gamma^-} = S_- q$,
$u_{\gamma^+} = S_+ q$ with $S_-\in\R^{|\gamma^-|\times|\gamma^-|}$, and hence
the boundary algebraic relation $u_{\gamma^+} = S_+S_-^{-1}u_{\gamma^-}$, the
discrete analogue of a boundary integral equation. (A double-layer variant with
second-kind, mesh-independent conditioning is available~\cite{xia2025geom}; we use
the single-layer kernel, which is always well defined, and report its
conditioning in \Cref{sec:analysis}.)

\subsection{Dirichlet closure by local basis functions}
\label{sec:dirichlet-closure}
The physical condition $u=g$ holds on the true boundary $\Gamma$, not on grid
points. For each $\gamma^-$ point we select one nearby boundary intersection
$x_i\in\Gamma$ (the closest grid-line crossing, found by Newton iteration with
a bracketed fallback), giving a square set of collocation equations.
\edit{The implementation selects a $3\times3$ stencil for each intersection
and evaluates tensor-product quadratic ($Q_2$) Lagrange basis functions
$\phi_{jk}$. When a stencil requires points outside the original boundary
neighborhood, the trace is extended to those points by a second-order
cardinal-direction extrapolation.} In operator form,
\begin{equation}\label{eq:phi-system}
\sum_{x_{jk}\in\gamma} u_{jk}\,\phi_{jk}(x_i) = g(x_i)
\;\Longrightarrow\;
P_+ u_{\gamma^+} + P_- u_{\gamma^-} = g,
\end{equation}
\edit{where $P_+$ and $P_-$ include the local collocation and extrapolation
maps.} Substituting \cref{eq:single-layer} yields the square density system
\begin{equation}\label{eq:density-system}
Bq=(P_+ S_+ + P_- S_-)\,q \;=\; g,
\end{equation}
of size $|\gamma^-|\times|\gamma^-|$. Using exactly one intersection per
$\gamma^-$ point (rather than a least-squares closure) preserves a square
boundary algebraic system with the same number of equations and density
unknowns. \edit{Although the local interpolant is quadratic, the implemented
five-point scheme and extrapolation are used as a globally second-order method.} The local-basis construction
and its treatment of small cuts are inherited from~\cite{xia2025geom}; no
additional small-cut stabilization is introduced here.

\subsection{Interior reconstruction and gradient}
\label{sec:reconstruction}
Direct evaluation of \cref{eq:single-layer} at every interior point would cost
$O(|\gamma^-|\,|M^+|)$. The interior field is instead recovered by the
difference-potentials route, which trades this dense sum for a single fast
Poisson solve on a regular box. Embed $\free$ in a rectangular auxiliary box
$\Omega^0\supset\free$ and form the grid function that carries the computed
boundary-neighborhood trace and vanishes elsewhere,
\begin{equation}\label{eq:dp-trace}
v(m) = \begin{cases} u_\gamma(m), & m\in\gamma,\\ 0, & \text{otherwise},\end{cases}
\end{equation}
where $u_\gamma=(u_{\gamma^+},u_{\gamma^-})$ is supplied by the layer potential.
Solve the auxiliary problem
\begin{equation}\label{eq:dp-aux}
[Aw](m) = \begin{cases} 0, & m\in M^+,\\ [Av](m), & m\in M^-,\end{cases}
\qquad w=0 \ \text{on}\ \partial\Omega^0 ,
\end{equation}
\edit{with homogeneous Dirichlet data on the finite auxiliary box. The code
diagonalizes the five-point operator by a two-dimensional type-I discrete sine
transform; the transform is implemented through FFTs of length $2(N_g+1)$ in
each direction, where $N_g$ denotes the number of grid points in one coordinate direction of the auxiliary box. This is a finite-box Dirichlet inverse. A free-space
convolution obtained by padding the physical field is also viable.} Then $w$ is
discrete-harmonic on $M^+$ and matches the trace, $\operatorname{Tr}_\gamma
w=u_\gamma$ (from Difference Potentials Theory~\cite{ryaben2012method}), so $w$ reproduces the interior field at every point of $M^+$ in
$O(N_g^2\log N_g)$ work, \emph{independent of the geometry and of $|\gamma^-|$}. This
is the same sine-transform solver at every frame, regardless of how the obstacles move.

The gradient at interior grid points is formed by second-order centered
differences. Because the layer-potential representation supplies values at
\emph{all} points of $\gamma^\pm$, the centered stencil is well defined
everywhere in $M^+$ with no one-sided or least-squares cut-cell
reconstruction---a structural advantage for the gradient accuracy. Off-grid
gradients are bilinearly interpolated. \edit{Near an unfitted boundary, this
interpolation can mix nodal gradients from different sides of the mask; the
centered nodal stencil therefore does not by itself imply a uniformly
second-order off-grid gradient.}

The overall computational pipeline is summarized by the following operator-level schematic,
which emphasizes the structural separation between geometry processing, boundary
algebra, and sine-transform reconstruction:
\begin{equation}\label{eq:pipeline-schematic}
\begin{aligned}
\mbox{Level set }\Psi
&\longrightarrow (M^+,M^-,\gamma^+,\gamma^-) \\
&\longrightarrow \text{boundary intersections and } P^\pm \\
&\longrightarrow (P^+S_+ + P^-S_-)q = g \\
&\longrightarrow u_\gamma
\longrightarrow \text{sine-transform reconstruction}
\longrightarrow u,\grad_h u .
\end{aligned}
\end{equation}

\section{Moving-geometry decomposition and block reduction}
\label{sec:method-dynamic}
\begin{revisionblock}
Stationary geometry contributes reusable grid, kernel, and reconstruction data,
whereas moving obstacles change the boundary neighborhood swept by their level
sets. We use two complete re-solve workflows. Algorithm~\ref{alg:full-trace}
recomputes and solves the full trace system at every geometry.
Algorithm~\ref{alg:block-update} works with the density
system, separates its unknowns into static and dynamic sets, and reuses one
factorization of the invariant static block. Both algorithms produce the same
discrete Dirichlet solution; they differ in what is assembled and factored when
the geometry changes.

\subsection{Algorithm 1: full-trace recomputation}

For the geometry at time $t$, define the density matrix
\begin{equation}\label{eq:B-definition}
B(t):=P_+(t)S_+(t)+P_-(t)S_-(t),
\qquad B(t)q(t)=g(t).
\end{equation}
The trace unknown $y(t)=u_{\gamma^-}(t)=S_-(t)q(t)$ gives the equivalent
system
\begin{equation}\label{eq:resolve-bae}
C(t)y(t)=\bigl(P_+(t)S_+(t)S_-^{-1}(t)+P_-(t)\bigr)y(t)=g(t),
\qquad C(t)=B(t)S_-^{-1}(t).
\end{equation}
The factor $S_-^{-1}(t)$ acts as a right preconditioner, so $C(t)$ is the
better-conditioned system in the present tests. Its construction contains the
inverse of a geometry-dependent global layer matrix, however, and therefore
does not expose a static subblock that can be reused without additional
analysis. Algorithm~\ref{alg:full-trace} consequently rebuilds the complete
trace system for every frame.

\begin{algorithm}[htbp]
\caption{Full-trace recomputation for one geometry.}
\label{alg:full-trace}
\begin{algorithmic}[1]
\Require Updated obstacle level sets $\phi_k(\cdot,t)$.
\State Reclassify $M^\pm(t)$ and $\gamma^\pm(t)$ and compute the current boundary--grid intersections.
\State Assemble the interpolation operators $P_\pm(t)$ and layer matrices $S_\pm(t)$.
\State Form the full trace matrix $C(t)$ in \cref{eq:resolve-bae} and the boundary-data vector $g(t)$.
\State Solve $C(t)y(t)=g(t)$ and recover $q(t)=S_-^{-1}(t)y(t)$.
\State Evaluate the layer trace, reconstruct $u(t)$ by the sine-transform difference-potential solve, and form $\grad_h u(t)$.
\State \Return $u(t)$ and $\grad_h u(t)$.
\end{algorithmic}
\end{algorithm}

This workflow is geometry-general and provides the reference solution in the
moving-domain comparisons. Because it assembles $S_\pm(t)$ and $C(t)$ anew, it
does not attempt to reuse boundary algebra from the preceding frame.

\subsection{Algorithm 2: cached static--dynamic block update}

To reuse the fixed part of the boundary algebra, partition the current density
unknowns using a prescribed envelope of every possible moving obstacle:
\begin{equation}\label{eq:gamma-partition}
\gamma^-(t) = \gamma^-_{\mathrm{s}} \cup \gamma^-_{\mathrm{d}}(t), \qquad
\gamma^-_{\mathrm{s}} \cap \gamma^-_{\mathrm{d}}(t) = \varnothing.
\end{equation}
The static set contains unknowns whose collocation stencils belong only to
permanently fixed boundaries, such as the outer boundary and the goal. The
dynamic set contains the current $\gamma^-$ points inside the swept obstacle
envelope, so its membership and size may change with $t$. We assume that the
swept envelope does not intersect a fixed-boundary interpolation neighborhood
and that the static coordinates and their ordering remain unchanged across
frames.

With the ordering in \cref{eq:gamma-partition}, the density system becomes
\begin{equation}\label{eq:schur-partition}
\begin{bmatrix}
  B_{\mathrm{ss}} & B_{\mathrm{sd}}(t) \\
  B_{\mathrm{ds}}(t) & B_{\mathrm{dd}}(t)
\end{bmatrix}
\begin{bmatrix} q_{\mathrm{s}} \\ q_{\mathrm{d}} \end{bmatrix}
=
\begin{bmatrix} g_{\mathrm{s}} \\ g_{\mathrm{d}}(t) \end{bmatrix}.
\end{equation}
Here $B_{\mathrm{ss}}$ is invariant because both its collocation rows and its
density-source columns are attached to fixed boundaries. Eliminating
$q_{\mathrm s}$ gives
\begin{equation}\label{eq:schur-complement}
\bigl(B_{\mathrm{dd}}-B_{\mathrm{ds}}B_{\mathrm{ss}}^{-1}B_{\mathrm{sd}}\bigr)
q_{\mathrm{d}}
=g_{\mathrm{d}}-B_{\mathrm{ds}}B_{\mathrm{ss}}^{-1}g_{\mathrm{s}},
\end{equation}
followed by
$q_{\mathrm{s}}=B_{\mathrm{ss}}^{-1}(g_{\mathrm{s}}-B_{\mathrm{sd}}q_{\mathrm{d}})$.
The LU factorization of $B_{\mathrm{ss}}$ is computed only on the first frame.
If $n_{\mathrm s}=|\gamma^-_{\mathrm s}|$ and
$n_{\mathrm d}(t)=|\gamma^-_{\mathrm d}(t)|$, the subsequent dense block
construction and solve cost
\begin{equation}\label{eq:schur-cost}
O(n_{\mathrm s}^2n_{\mathrm d}
  +n_{\mathrm s}n_{\mathrm d}^2+n_{\mathrm d}^3).
\end{equation}
The dynamic fraction controls the reduction, but the speedup is not a simple
cubic factor $(n/n_{\mathrm d})^3$ because the dynamic variables remain densely
coupled to all static variables.

\begin{algorithm}[htbp]
\caption{Cached static--dynamic block update.}
\label{alg:block-update}
\begin{algorithmic}[1]
\Statex \textbf{One-time initialization}
\State Tabulate the lattice Green's function and initialize the sine-transform reconstruction solver.
\State Prescribe the swept obstacle envelope and the static/dynamic ordering.
\State On the first frame, cache the static coordinates and collocation stencils; assemble and factor $B_{\mathrm{ss}}$.
\Statex
\Statex \textbf{Update at time $t$}
\Require Updated obstacle level sets $\phi_k(\cdot,t)$.
\State Rebuild the current classification, intersections, and local collocation data, and verify the static coordinate set.
\State Assemble $B_{\mathrm{sd}}(t)$, $B_{\mathrm{ds}}(t)$, $B_{\mathrm{dd}}(t)$, and the current right-hand side; do not assemble $C(t)$.
\State Apply the cached LU factors in \cref{eq:schur-complement}, solve for $q_{\mathrm d}(t)$, and recover $q_{\mathrm s}(t)$.
\State Evaluate the layer trace, reconstruct $u(t)$ by the sine-transform difference-potential solve, and form $\grad_h u(t)$.
\State \Return $u(t)$ and $\grad_h u(t)$.
\end{algorithmic}
\end{algorithm}

\edit{The cache/update split used by the implementation is summarized in
\Cref{tab:static-dynamic}: only the right column and the blocks coupled to it
are refreshed after initialization.}

\begin{table}[htbp]
\centering
\caption{Static/dynamic decomposition used by Algorithm~\ref{alg:block-update}.
The static column is cached once; the dynamic column is refreshed when the
obstacle geometry changes.}
\label{tab:static-dynamic}
\begin{tabular}{ll}
  \toprule
  \textbf{Static (precomputed once)} & \textbf{Dynamic (recomputed per frame)} \\
  \midrule
  Cartesian grid $h\Z^2$ and auxiliary box & Boundary classification $M^\pm(t),\gamma^\pm(t)$ \\
  Lattice Green's function table $G(m)$ & Boundary--grid intersections \\
  Sine-transform reconstruction solver & Moving-boundary interpolation weights \\
  Envelope, static coordinates, and ordering & Dynamic coordinates inside the envelope \\
  Fixed collocation stencils and $g_{\mathrm s}$ & Moving-boundary data $g_{\mathrm d}(t)$ \\
  $B_{\mathrm{ss}}$ and its LU factorization & $B_{\mathrm{sd}}(t),B_{\mathrm{ds}}(t),B_{\mathrm{dd}}(t)$ \\
  \bottomrule
\end{tabular}
\end{table}

\subsection{Relationship between the two workflows}

Algorithms~\ref{alg:full-trace} and \ref{alg:block-update} differ only in the
boundary solve and its reusable data. Algorithm~\ref{alg:full-trace} recomputes the globally
preconditioned trace matrix and is used as the full-system reference.
Algorithm~\ref{alg:block-update} avoids constructing $C(t)$, reuses the factorization of
$B_{\mathrm{ss}}$, and assembles only blocks coupled to the swept obstacle
neighborhood. Under the nonsingularity assumptions of
\Cref{prop:schur-equivalence}, the two boundary solutions are algebraically
equivalent.

The present implementation of Algorithm~\ref{alg:block-update} still rebuilds the full-grid geometry
classification, intersections, and local collocation data each frame; the
numerical profiling separates this geometry cost from the reduced boundary
algebra.

\begin{remark}[Receding-horizon use]
Either workflow can serve a sensing-limited agent: solve the field on the
currently known free space, advance the agent, update the obstacle level sets
from new observations, and solve again. Algorithm~\ref{alg:block-update} is applicable when a fixed
envelope of possible geometry changes is available.
\end{remark}

The numerical results below use Algorithm~\ref{alg:full-trace} as the full-system reference and
Algorithm~\ref{alg:block-update} for cached moving-geometry updates. They test algebraic equivalence,
one-time static factorization, and the remaining per-frame geometry cost.
\end{revisionblock}

\section{Analysis: solvability, convergence, conditioning, complexity}
\label{sec:analysis}

This section provides the theoretical underpinning for the method
developed in \Cref{sec:method,sec:method-dynamic}. 

\subsection{Solvability}

The density system \cref{eq:density-system} couples three operators on the
boundary neighborhood: the single-layer matrices $S_\pm$, defined by
$u_{\gamma^\pm}=S_\pm q$ for the single-layer potential $\mathcal S q$ of
\cref{eq:single-layer}, and the sparse local collocation operators $P_\pm$ of
\cref{eq:phi-system}. We make the underlying discrete Dirichlet problem explicit
and then reduce solvability of \cref{eq:density-system} to two transparent
conditions.

\begin{definition}[Discrete interior Dirichlet problem]
Given boundary data $\varphi$ on $\gamma^-$, the discrete-harmonic extension is
the grid function $w$ on $N^+:=M^+\cup\gamma^-$ with $(Aw)(m)=0$ for
$m\in M^+$ and $w=\varphi$ on $\gamma^-$, where $A$ is the five-point Laplacian.
\end{definition}

\begin{lemma}[Discrete maximum principle]\label{lem:dmp}
If $M^+$ is finite, nonempty, and connected in the grid graph, and $(Aw)(m)=0$
for all $m\in M^+$, then $\min_{\gamma^-}w\le w(m)\le\max_{\gamma^-}w$ for every
$m\in M^+$. In particular the discrete interior Dirichlet problem has a unique
solution, and $w\equiv0$ on $N^+$ whenever $\varphi\equiv0$.
\end{lemma}

\begin{proof}
For $m\in M^+$, $(Aw)(m)=0$ reads $w(m)=\tfrac14\sum_{|e|=1}w(m+e)$: each interior
value is the average of its four axis-neighbors, which lie in $M^+\cup\gamma^-$.
Hence $w$ admits no strict interior extremum; if the maximum over
$M^+$ were attained at some $m_0\in M^+$, equality in the average
forces the four neighbors to share that value, and \rev{connectedness in the
grid graph} propagates it through $M^+$ and to the adjacent points of $\gamma^-$, so the maximum is also
attained on $\gamma^-$; likewise for the minimum. Uniqueness follows by applying
the bound to the difference of two solutions, and the square system then also has
existence.
\end{proof}

\rev{Thus the five-point interior operator supplies a discrete analogue of the
continuous maximum principle: nodal values cannot create a strict interior
extremum. This statement does not, by itself, control extrema of the off-grid
interpolant or guarantee that a numerically integrated trajectory avoids all
discrete critical points.}

By \cref{lem:dmp} the extension is well defined; let
$E:\R^{\gamma^-}\to\R^{\gamma^+}$ return its values on $\gamma^+$, i.e.\
$w_{\gamma^+}=E\varphi$, and define the \emph{Dirichlet collocation operator}
\[
C:=P_+E+P_-\in\R^{|\gamma^-|\times|\gamma^-|}.
\]

\begin{proposition}[Solvability of the density system]\label{prop:solvability}
Assume \rev{(H1) each $\gamma^-$ point is assigned exactly one boundary
intersection lying on a Cartesian grid edge having that point as an endpoint,
as constructed by the edge-based Newton search of
\Cref{sec:dirichlet-closure},} so the collocation system is square, and that $M^+$ is
finite, nonempty, and connected. Then
\[
P_+S_++P_-S_- \;=\; C\,S_-,
\]
so the density matrix is nonsingular if and only if both $C$ and $S_-$ are.
In that case \cref{eq:density-system} has a unique solution $q$, and
$\mathcal S q$ is the unique grid function that is discrete-harmonic on $M^+$ and
whose local $Q_2$ interpolant, including the stated extrapolation, matches the
data $g$ at the collocation points.
\end{proposition}

\begin{proof}
For any density $q$, $\mathcal S q$ satisfies $A(\mathcal S q)=q$ on $\gamma^-$
and $A(\mathcal S q)=0$ on $M^+$, since the sources lie in $\gamma^-\subset M^-$.
Thus $\mathcal S q$ restricted to $N^+$ is discrete-harmonic on $M^+$
with $\gamma^-$-trace $S_-q$; by the uniqueness in \cref{lem:dmp} it coincides
with the extension of $S_-q$, so its $\gamma^+$-values satisfy $S_+q=E\,S_-q$.
As this holds for every $q$, $S_+=ES_-$, whence
$P_+S_++P_-S_-=(P_+E+P_-)S_-=CS_-$. Both factors are square of order
$|\gamma^-|$, so $\det(CS_-)=\det C\,\det S_-$, which gives the equivalence. When
both are nonsingular $q$ is unique, and $C\,(S_-q)=g$ exhibits $\mathcal S q$ as
the discrete-harmonic field whose interpolated trace equals $g$.
\end{proof}

\begin{remark}[Well-posedness versus representation]\label{rem:wp-rep}
The two factors play different roles. Nonsingularity of $C$ is the genuine
well-posedness of the discretized boundary value problem: $C\,w_{\gamma^-}=g$
determines the discrete-harmonic field directly from the imposed data, with no
reference to a layer density. Nonsingularity of $S_-$ is a property of the
single-layer \emph{representation}: a null density $q_0$ with $S_-q_0=0$
generates, by \cref{lem:dmp}, a potential vanishing on $N^+$, hence a
gauge mode contributing nothing to the interior field rather than a defect of the
field itself.
\end{remark}

\begin{remark}[Unisolvence of the collocation closure]\label{rem:uniso}
\Cref{prop:solvability} is conditional on nonsingularity of $C$ and
$S_-$. The local-basis difference-potentials analysis in
\cite{ryaben2012method,xia2025geom} gives the relevant consistency and
unisolvence framework. In the present paper these conditions are checked
numerically through the observed conditioning of the assembled systems; a new
uniform lower bound for arbitrary moving cut configurations is not claimed.
\end{remark}

\subsection{Consistency and observed convergence}
\label{sec:consistency}
The interior operator is the standard second-order five-point Laplacian, and the
Dirichlet closure \cref{eq:phi-system} has second-order interpolation error for
smooth data. The corresponding second-order convergence of the potential is
established for the underlying local-basis difference-potentials method in
\cite{xia2025geom}. For the reconstructed gradient, centered differencing of a
generic second-order nodal approximation does not by itself imply a second-order
maximum-norm estimate. We therefore treat gradient accuracy as a numerical
verification target rather than a proved result. The manufactured-solution tests
in \Cref{sec:selfconv-mfg} show near-second-order bulk convergence for the
configurations considered here.

\subsection{Conditioning}
\edit{The single-layer density formulation becomes increasingly ill-conditioned
as the grid is refined, so its conditioning must be checked.
\Cref{tab:conditioning-refinement} reports spectral condition numbers at the
midpoint geometry of the translating-circle test. The trace matrix remains well
conditioned, whereas $S_-$, $B$, and $B_{\mathrm{ss}}$ grow with refinement.
The reduced density-block Schur complement is substantially better conditioned
than the full density matrix, although its condition number also grows.}

\begin{table}[htbp]
\centering
\small
\caption{Spectral conditioning at the midpoint of the translating-circle
trajectory. Here $S_B=B_{\mathrm{dd}}-B_{\mathrm{ds}}
B_{\mathrm{ss}}^{-1}B_{\mathrm{sd}}$.}
\label{tab:conditioning-refinement}
\begin{tabular}{lrrrrrrr}
  \toprule
  $\ell$ & $N_g^2$ & $n$ & $\kappa_2(S_-)$ & $\kappa_2(B)$ & $\kappa_2(C)$ & $\kappa_2(B_{\mathrm{ss}})$ & $\kappa_2(S_B)$ \\
  \midrule
  5 & $31^2$  & 117  & $4.61\times10^2$ & $7.16\times10^2$ & 2.93 & $4.78\times10^2$ & $2.22\times10^1$ \\
  6 & $63^2$  & 245  & $1.12\times10^3$ & $1.78\times10^3$ & 2.35 & $1.64\times10^3$ & $5.13\times10^1$ \\
  7 & $127^2$ & 500  & $2.60\times10^3$ & $5.02\times10^3$ & 2.99 & $4.05\times10^3$ & $1.51\times10^2$ \\
  8 & $255^2$ & 1012 & $5.89\times10^3$ & $2.14\times10^4$ & 5.46 & $1.95\times10^4$ & $4.14\times10^2$ \\
  \bottomrule
\end{tabular}
\end{table}

\edit{The relative Frobenius residual in the independently assembled identity
$B=CS_-$ is at most $3.82\times10^{-15}$ across these levels. The largest
observed density-block residual in the moving tests is
$1.07\times10^{-14}$. These results support the present direct solves through
$\ell=8$, but they also motivate a well-conditioned or iteratively
preconditioned block formulation at substantially finer resolutions.}

\subsection{Complexity}
Let $n=|\gamma^-|=n_{\mathrm s}+n_{\mathrm d}$ and let $N_g$ denote the number
of grid points in one coordinate direction of the auxiliary box. Full dense
assembly of the density matrix costs $O(n^2)$ kernel-table accesses, a dense
matrix--vector product or triangular solve costs $O(n^2)$, and a direct
factorization costs $O(n^3)$. The sine-transform reconstruction costs
$O(N_g^2\log N_g)$ in two dimensions. Because
$n$ scales with boundary length rather than domain area, the boundary algebra is
attractive for perimeter-limited geometries, although dense assembly is already
the measured bottleneck at the present two-dimensional sizes.

After the one-time $O(n_{\mathrm s}^3)$ factorization of $B_{\mathrm{ss}}$, the
Schur construction and solve have the per-frame cost stated in
\cref{eq:schur-cost}. \edit{The implemented block assembler refreshes}
$O(n_{\mathrm s}n_{\mathrm d}+n_{\mathrm d}^2)$ entries rather than all
$O(n^2)$ entries. At larger boundary sizes, hierarchical or fast-multipole
compression~\cite{hackbusch1999sparse,greengard1987fast,liska2014parallel} is a
natural route for accelerating both matrix construction and matrix--vector
products. \edit{Kernel compression and a swept-band geometry update are not
implemented in the present work.}

\subsection{Cost model for geometry motion}
\label{sec:analysis-resolve}
If an obstacle boundary of perimeter $P$ moves by a distance $\delta$, the grid
points whose classification or interpolation stencil can change lie in a swept
band of area $O(P\delta)$. On a Cartesian grid, the corresponding geometry
update involves $O(P\delta/h^2)$ grid points, capped by the size of the boundary
neighborhood. This locality applies to classification, intersection finding, and
interpolation data. The dense boundary algebra is different: each dynamic
unknown interacts with all static unknowns, so a block refresh contains
$2n_{\mathrm s}n_{\mathrm d}+n_{\mathrm d}^2$ entries. Under rigid translation,
$n_{\mathrm d}$ and the total boundary size are expected to vary only mildly,
but position robustness must be assessed numerically rather than inferred from
this count alone.

\begin{remark}
In three dimensions the boundary unknown count grows with surface area and can
reach $O(10^4)$ even for moderate grids. Dense factorization and matrix--vector
products then become prohibitive. The static-block reduction remains useful when
the dynamic surface is a small fraction of the total boundary, but the cross
terms in \cref{eq:schur-cost} remain. \edit{Practical three-dimensional
performance will therefore require kernel compression and a local geometry
update in addition to the implemented block refresh; these extensions are left
for future work.}
\end{remark}

\subsection{Algebraic equivalence of the block reduction}
\label{sec:schur-equivalence}
\begin{revisionblock}
\begin{proposition}[Schur equivalence]\label{prop:schur-equivalence}
Suppose $B_{\mathrm{ss}}$ is nonsingular and define
$S_B=B_{\mathrm{dd}}-B_{\mathrm{ds}}B_{\mathrm{ss}}^{-1}B_{\mathrm{sd}}$.
Then the full block matrix in \cref{eq:schur-partition} is nonsingular if and
only if $S_B$ is nonsingular. In that case, solving
\cref{eq:schur-complement} and recovering $q_{\mathrm s}$ gives exactly the
solution of the full system in exact arithmetic.
\end{proposition}
\begin{proof}
Suppressing the time argument, block elimination gives the factorization
\[
\begin{bmatrix}B_{\mathrm{ss}}&B_{\mathrm{sd}}\\
B_{\mathrm{ds}}&B_{\mathrm{dd}}\end{bmatrix}
=
\begin{bmatrix}I&0\\B_{\mathrm{ds}}B_{\mathrm{ss}}^{-1}&I\end{bmatrix}
\begin{bmatrix}B_{\mathrm{ss}}&B_{\mathrm{sd}}\\0&S_B\end{bmatrix}.
\]
The first factor is always nonsingular, and the second is nonsingular exactly
when both diagonal blocks are nonsingular. Forward substitution in this
factorization yields \cref{eq:schur-complement} and the stated recovery formula.
\end{proof}
\end{revisionblock}

\edit{The same-grid experiments in \Cref{sec:schur-comparison} test this exact
algebraic statement on explicitly specified translating and topology-changing
geometries.}

\section{Numerical verification and performance}
\label{sec:self-convergence}

\edit{We organize the numerical study from discretization accuracy to
moving-geometry equivalence and performance. We first assess the convergence
properties of the LGF-BAE discretization on two complementary
geometries: a circular obstacle without re-entrant corners in the free domain, and
a cross-shaped obstacle whose convex corners create re-entrant corners in the surrounding free space.  Both cases use the domain $[-1,1]^2$ with Dirichlet
conditions $u=0$ at a small goal region and $u=1$ on all other boundaries.
Self-convergence is measured against an $\ell=8$ reference solution on nested
grids. Errors are reported at common grid points more than $3h$ from any
boundary. Because this exclusion distance shrinks with refinement, the resulting
gradient orders should be interpreted as mesh-dependent bulk diagnostics rather
than uniform interior estimates. We then specify the translating and
topology-changing configurations used to test same-grid agreement, re-solve
cost, and component-level timing.}

\edit{\paragraph{Algorithm assignment.}
The smooth-geometry, corner-geometry, manufactured-solution, and subgrid-position
tests use Algorithm~\ref{alg:full-trace}. The moving-circle convergence,
same-grid agreement, and profiling tests compare
Algorithm~\ref{alg:full-trace} with Algorithm~\ref{alg:block-update}. All
dynamic trajectory figures in
\Cref{sec:app-dynamic,sec:topology-change,sec:appearing-obstacle} are generated
with Algorithm~\ref{alg:block-update}.}

\subsection{Smooth geometry: circular obstacle}

The first test places a circular obstacle of radius $0.3$ centered
at the origin, with a goal disk of radius $0.08$ at $(-0.6,0)$.  The obstacle and goal boundaries are smooth and the free domain has no
re-entrant corners. The outer boundary remains the square $[-1,1]^2$. We expect
second-order potential convergence from the discretization; gradient convergence
is assessed numerically. \edit{The two representative fields are shown side by
side in \Cref{fig:field-comparison}, and \Cref{tab:selfconv-disk} reports the
corresponding bulk self-convergence errors.}

\begin{figure}[htbp]
\centering
\subcaptionbox{Smooth disk}
  {\includegraphics[width=0.35\linewidth]{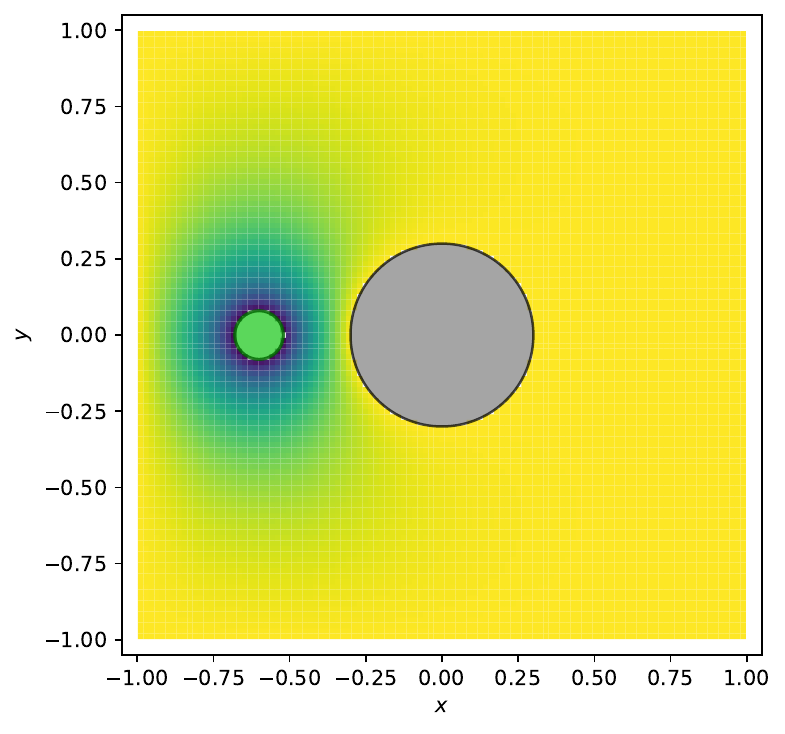}}
~
\subcaptionbox{Cross-shaped obstacle}
  {\includegraphics[width=0.35\linewidth]{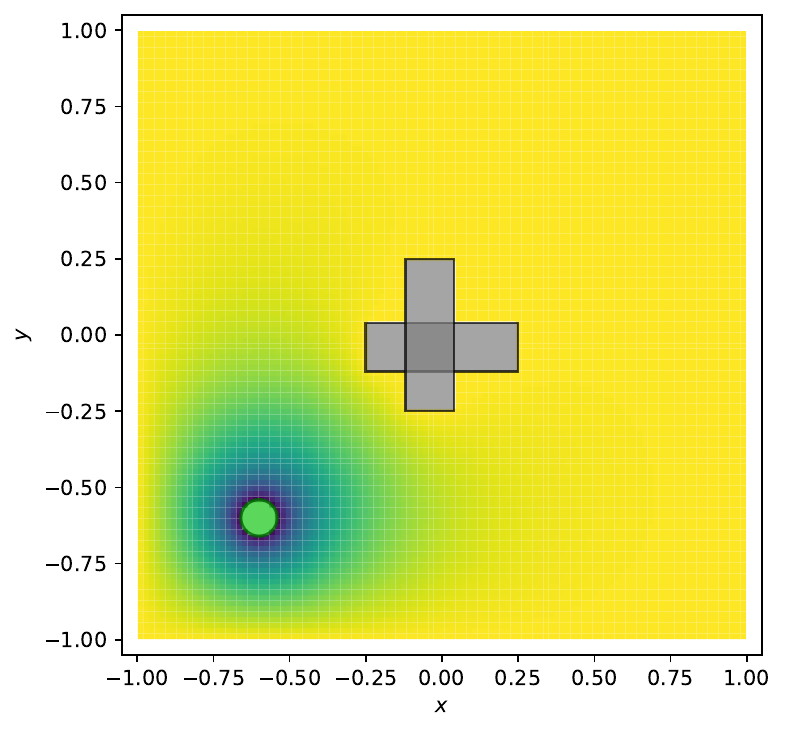}}
\caption{Harmonic potential fields ($\ell=7$).  Left: circular
obstacle with goal at $(-0.6,0)$.  Right: cross-shaped obstacle near the
origin; the convex obstacle corners correspond to re-entrant corners of the free domain.  $u=0$ at
the goal (green), $u=1$ on obstacles and outer boundary.}
\label{fig:field-comparison}
\end{figure}

\begin{table}[htbp]
\centering
\small
\caption{Self-convergence on the smooth disk geometry.
Reference: $\ell=8$ ($h\approx0.009$, $255\times255$ grid).
Bulk $L^\infty$ errors.}
\label{tab:selfconv-disk}
\begin{tabular}{lccccc}
  \toprule
  $\ell$ & $h$ & $\|u-u_8\|_\infty$ & order & $\|\nabla u-\nabla u_8\|_\infty$ & order \\
  \midrule
  5 & 0.072 & $1.74\times10^{-2}$ & --   & $9.23\times10^{-2}$ & -- \\
  6 & 0.036 & $2.50\times10^{-3}$ & 2.80 & $4.64\times10^{-2}$ & 0.99 \\
  7 & 0.018 & $6.03\times10^{-4}$ & 2.05 & $2.37\times10^{-2}$ & 0.97 \\
  \bottomrule
\end{tabular}
\end{table}

\rev{As reported in \Cref{tab:selfconv-disk}, the potential exhibits
approximately second-order self-convergence. The
gradient differences decrease at approximately first order in this experiment,
but those slopes are not discretization-order estimates: they compare two
discrete gradients over a mesh-dependent exclusion region. The primary gradient
order estimate is therefore the exact-error manufactured-solution test in
\Cref{sec:selfconv-mfg}, which gives orders $1.82$--$1.95$ on the smooth disk.}

\subsection{Corner geometry: cross-shaped obstacle}

The second test uses a cross-shaped obstacle formed by two perpendicular
rectangles. Each convex corner of the obstacle is a re-entrant corner of the free
domain with local opening angle $3\pi/2$. A reduction of convergence rate is expected.

\begin{table}[htbp]
\centering
\small
\caption{Self-convergence with the cross-shaped obstacle. The data are pre-asymptotic and are not used to estimate the corner-singularity exponent.}
\label{tab:selfconv-cross}
\begin{tabular}{lccccc}
  \toprule
  $\ell$ & $h$ & $\|u-u_8\|_\infty$ & order & $\|\nabla u-\nabla u_8\|_\infty$ & order \\
  \midrule
  5 & 0.072 & $8.93\times10^{-3}$ & --   & $8.60\times10^{-2}$ & -- \\
  6 & 0.036 & $6.43\times10^{-3}$ & 0.47 & $8.59\times10^{-2}$ & 0.00 \\
  7 & 0.018 & $2.33\times10^{-3}$ & 1.46 & $3.82\times10^{-2}$ & 1.17 \\
  \bottomrule
\end{tabular}
\end{table}

At $\ell=5$ the thinner bars are represented by only a few cells, and the
observed orders vary substantially across refinement levels in \Cref{tab:selfconv-cross}. They show that the
method remains stable and that gradient errors are larger than in the circular
case. 


\subsection{Manufactured solution verification}
\label{sec:selfconv-mfg}

To confirm that the self-convergence rates reflect the true
discretization order, we compare against the exact harmonic solution
$u(x,y)=e^x\cos y$ ($\Delta u=0$, all derivatives non-zero).  Dirichlet
data equal to the exact trace are imposed on all boundaries for both the
smooth disk and the cross-shaped obstacle.  Errors are evaluated directly against the exact solution and exact gradient at the grid points used in each norm.
\edit{\Cref{tab:mfg-both} reports the bulk errors, while
\Cref{tab:mfg-gradient-regions} resolves the gradient error over the full
interior and the first two interior grid layers.}

\begin{table}[htbp]
\centering
\small
\caption{Manufactured solution convergence on the smooth disk
and the cross-shaped obstacle.  Both show near-second-order bulk convergence for the potential and gradient.}
\label{tab:mfg-both}
\begin{tabular}{lccccccccc}
  \toprule
  & \multicolumn{4}{c}{Smooth disk} & \multicolumn{4}{c}{Cross-shaped obstacle} \\
  \cmidrule(lr){2-5}\cmidrule(lr){6-9}
  $\ell$ & $h$ & $\|u-u_e\|_\infty$ & o. & $\|\nabla u-\nabla u_e\|_\infty$ & o. & $\|u-u_e\|_\infty$ & o. & $\|\nabla u-\nabla u_e\|_\infty$ & o. \\
  \midrule
  5 & 0.072 & $1.08\times10^{-4}$ & -- & $1.60\times10^{-3}$ & -- & $1.31\times10^{-4}$ & -- & $1.54\times10^{-3}$ & -- \\
  6 & 0.036 & $2.69\times10^{-5}$ & 2.01 & $4.54\times10^{-4}$ & 1.82 & $3.15\times10^{-5}$ & 2.06 & $4.52\times10^{-4}$ & 1.77 \\
  7 & 0.018 & $6.81\times10^{-6}$ & 1.98 & $1.27\times10^{-4}$ & 1.84 & $8.00\times10^{-6}$ & 1.98 & $1.26\times10^{-4}$ & 1.84 \\
  8 & 0.009 & $1.77\times10^{-6}$ & 1.95 & $3.27\times10^{-5}$ & 1.95 & $2.07\times10^{-6}$ & 1.95 & $3.27\times10^{-5}$ & 1.95 \\
  \bottomrule
\end{tabular}
\end{table}

\begin{table}[htbp]
\centering
\footnotesize
\setlength{\tabcolsep}{3pt}
\caption{Exact Euclidean gradient errors over all interior nodes and the first
two interior grid layers. The discrete $L^2$ norm includes the area weight
$h^2$; ``angle'' is the maximum normalized-gradient angle in the two-layer
region, in radians.}
\label{tab:mfg-gradient-regions}
\begin{tabular}{llrrrrrrrrr}
  \toprule
  Geometry & $\ell$ & $L^\infty_{\rm all}$ & o. & $L^2_{\rm all}$ & o. & $L^\infty_{\rm 2layer}$ & o. & $L^2_{\rm 2layer}$ & o. & Angle \\
  \midrule
  Disk & 5 & $2.02\times10^{-3}$ & --   & $2.03\times10^{-3}$ & --   & $2.02\times10^{-3}$ & --   & $1.28\times10^{-3}$ & --   & $9.44\times10^{-4}$ \\
       & 6 & $5.38\times10^{-4}$ & 1.91 & $5.22\times10^{-4}$ & 1.96 & $5.38\times10^{-4}$ & 1.91 & $2.38\times10^{-4}$ & 2.43 & $2.45\times10^{-4}$ \\
       & 7 & $1.41\times10^{-4}$ & 1.93 & $1.33\times10^{-4}$ & 1.98 & $1.41\times10^{-4}$ & 1.93 & $4.33\times10^{-5}$ & 2.46 & $6.28\times10^{-5}$ \\
       & 8 & $3.63\times10^{-5}$ & 1.96 & $3.36\times10^{-5}$ & 1.98 & $3.63\times10^{-5}$ & 1.96 & $7.79\times10^{-6}$ & 2.47 & $1.60\times10^{-5}$ \\
  \midrule
  Cross & 5 & $2.01\times10^{-3}$ & --   & $2.04\times10^{-3}$ & --   & $2.01\times10^{-3}$ & --   & $1.25\times10^{-3}$ & --   & $1.04\times10^{-3}$ \\
        & 6 & $5.59\times10^{-4}$ & 1.85 & $5.24\times10^{-4}$ & 1.96 & $5.59\times10^{-4}$ & 1.85 & $2.34\times10^{-4}$ & 2.42 & $3.22\times10^{-4}$ \\
        & 7 & $1.47\times10^{-4}$ & 1.93 & $1.33\times10^{-4}$ & 1.98 & $1.47\times10^{-4}$ & 1.93 & $4.26\times10^{-5}$ & 2.45 & $8.90\times10^{-5}$ \\
        & 8 & $3.91\times10^{-5}$ & 1.91 & $3.36\times10^{-5}$ & 1.98 & $3.91\times10^{-5}$ & 1.91 & $7.68\times10^{-6}$ & 2.47 & $2.35\times10^{-5}$ \\
  \bottomrule
\end{tabular}
\end{table}

\edit{The full-interior gradient errors converge at orders 1.85--1.98 in the
last three refinements, so the near-second-order result is not an artifact of a
boundary exclusion band. The first-two-layer $L^\infty$ errors converge at the
same rates, while their area-weighted $L^2$ errors decrease faster because the
physical area of that layer shrinks with $h$. At $\ell=8$, the corresponding
two-layer potential $L^\infty$ errors are $2.29\times10^{-7}$ for the disk and
$2.75\times10^{-7}$ for the cross. We report these regional results without a
superconvergence claim.}

Both geometries show near-second-order bulk convergence for $u$ and
$\nabla u$. The cross-shaped boundary does not reduce the rate for this test
because the prescribed exact solution is harmonic and smooth in a neighborhood
of the entire computational box. The comparison with
\Cref{tab:selfconv-disk,tab:selfconv-cross} also shows that gradient
self-convergence is more sensitive to the choice of comparison region than the
exact-error experiment; no general superconvergence claim is made.

\subsection{Block update vs.\ full trace under motion}
\label{sec:selfconv-block}

\edit{The fourth test verifies that Algorithm~\ref{alg:block-update} preserves
the convergence behavior of Algorithm~\ref{alg:full-trace}. The outer domain is
$[-1,1]^2$, the goal is the radius-$0.05$ disk centered at $(-0.80,0.80)$,
and a circular obstacle of radius $0.15$ moves through five equally spaced
centers from $(0.40,0.20)$ to $(-0.40,-0.20)$;
at each position both algorithms solve the field at $\ell=5,6,7,8$. The
$\ell=8$ Algorithm~\ref{alg:full-trace} solution is the reference, and errors
are measured in the bulk as before. Algorithm~\ref{alg:block-update} uses a
capsule containing the complete obstacle
trajectory, giving a consistent static/dynamic partition at every frame. Its
static block is factored once at each refinement level. The paired convergence
results are reported in \Cref{tab:selfconv-block}.}

\begin{table}[htbp]
\centering
\small
\caption{\edit{Self-convergence of Algorithm~\ref{alg:full-trace} (F, full
trace) and Algorithm~\ref{alg:block-update} (B, block update) for a translating
obstacle, averaged over five positions.
Direct same-grid differences are summarized in the text.}}
\label{tab:selfconv-block}
\begin{tabular}{lccccc}
  \toprule
  $\ell$ & $h$ & $\|u_{\rm F}-u_8\|_\infty$ & $\|u_{\rm B}-u_8\|_\infty$ & $\|\nabla u_{\rm F}-\nabla u_8\|_\infty$ & $\|\nabla u_{\rm B}-\nabla u_8\|_\infty$ \\
  \midrule
  5 & 0.072 & $2.15\times10^{-2}$ & $2.15\times10^{-2}$ & $1.13\times10^{-1}$ & $1.13\times10^{-1}$ \\
  6 & 0.036 & $8.05\times10^{-3}$ & $8.05\times10^{-3}$ & $8.87\times10^{-2}$ & $8.87\times10^{-2}$ \\
  7 & 0.018 & $1.60\times10^{-3}$ & $1.60\times10^{-3}$ & $6.02\times10^{-2}$ & $6.02\times10^{-2}$ \\
  \bottomrule
\end{tabular}
\end{table}

\rev{The paired columns in \Cref{tab:selfconv-block} agree to three significant
figures at every resolution.
The direct same-grid Algorithm~\ref{alg:block-update}--Algorithm~\ref{alg:full-trace}
differences, evaluated separately
from the reference-error table, average approximately $10^{-10}$ for $u$ and
$10^{-9}$ for $\nabla u$, consistent with \Cref{prop:schur-equivalence} up to
solver and roundoff error.} Least-squares potential orders are 1.88 for both formulations. The much lower
gradient slope is not interpreted as an asymptotic order because only three
coarse levels are compared and the norm excludes a mesh-dependent $3h$ boundary
band. The nominal static fraction grows from 51\% at $\ell=6$ to approximately
90\% at $\ell=7$ and $\ell=8$. \edit{The one-time factorization count at every
level confirms that these columns were generated with the cached block update.}

\subsection{Same-grid agreement under translation and topology change}
\label{sec:schur-comparison}

\edit{\paragraph{Geometry configurations.}
Both comparisons use the outer square $[-1,1]^2$ and a circular goal of radius
$0.05$ centered at $(-0.80,0.80)$. In the translation test, one obstacle of
radius $0.15$ moves linearly from center $(0.30,0.15)$ to
$(-0.30,-0.15)$. We sample 21 positions at $\ell=7$ and seven positions at
$\ell=8$, and the dynamic set is prescribed by the capsule swept by this
circle. In the topology-change test, the obstacle is the union of two
radius-$0.22$ circles centered at $(0,\pm d)$. The offset $d$ increases from
$0$ to $0.34$ and returns to $0$ over 13 frames at $\ell=7$, changing the
union from one merged component to two separated components and back. Its
dynamic set lies in the fixed radius-$0.56$ circular envelope. Thus the
geometry, motion schedule, and reusable static boundary are fixed before the
two algorithms are compared.}

\edit{At every sampled geometry, Algorithm~\ref{alg:full-trace} supplies the
reference field and Algorithm~\ref{alg:block-update} supplies the cached block
solution. The execution order is rotated between frames to limit timing bias.
Algorithm~\ref{alg:block-update} factors $B_{\mathrm{ss}}$ once per run; its
maximum block residual remains near $10^{-14}$, and its field agrees with the
Algorithm~\ref{alg:full-trace} field to about $10^{-9}$.
\Cref{tab:schur-vs-full} reports the same-grid comparisons.}

\begin{table}[htbp]
\centering
\small
\caption{Same-grid algebraic comparisons for the configurations specified in
\Cref{sec:schur-comparison}. Field differences are maximum absolute differences
from the Algorithm~\ref{alg:full-trace} solution over all tested frames; the
residual is for Algorithm~\ref{alg:block-update}.}
\label{tab:schur-vs-full}
\begin{tabular}{lcccc}
  \toprule
  Case & Frames & $(n_{\mathrm s},n_{\mathrm d})$ & Algorithm~\ref{alg:block-update} diff. & Block residual \\
  \midrule
  Translation, $\ell=7$ & 21 & $(456,44\text{--}46)$ & $1.10\times10^{-9}$ & $6.33\times10^{-15}$ \\
  Translation, $\ell=8$ & 7 & $(920,92\text{--}93)$ & $1.01\times10^{-9}$ & $1.07\times10^{-14}$ \\
  Merge--split--merge, $\ell=7$ & 13 & $(456,68\text{--}136)$ & $1.03\times10^{-9}$ & $6.22\times10^{-15}$ \\
  \bottomrule
\end{tabular}
\end{table}

\edit{For the topology-change schedule, the dynamic fraction ranges from
13.0\% to 23.0\%, and the refreshed blocks contain 24.3--40.7\% of the full
dense entries. A separate diagnostic recomputes every cached static stencil at
each frame and confirms its invariance; this diagnostic is excluded from the
performance timing. The reduction in dynamic unknowns applies only to the
final reduced solve: applying the cached factor to the coupled blocks and
forming the Schur complement must also be counted, so no cubic speedup is
claimed. Together with the moving-circle convergence test in
\Cref{sec:selfconv-block}, these results verify both discretization-level and
same-grid agreement of the two algorithms.}

\subsection{Subgrid-position robustness}
\label{sec:position-robustness}

To assess sensitivity to the relative grid position, a small circular obstacle
($r=0.06$) was translated diagonally through $30$ positions over a displacement
of $0.015625\approx0.87h$ at $\ell=7$. At each position we recorded
$|\gamma^-|$, the dynamic-set size $n_d$, $\kappa_2(S_-)$, and total solve time.
\edit{\Cref{tab:position-robustness} summarizes the resulting variation.}

\begin{table}[htbp]
  \centering
  \small
  \caption{Subgrid-position robustness over 30 positions and a diagonal
  coordinate displacement $\Delta x=\Delta y=0.015625\approx0.87h$.
  CV = coefficient of variation.}
  \label{tab:position-robustness}
  \begin{tabular}{lcccc}
    \toprule
    Quantity & Mean & Std & Min/Max & CV \\
    \midrule
    $|\gamma^-|$           & 482.37  & 0.48 & 482/483 & 0.0010 \\
    $n_d$ (dynamic points) & 16.37   & 0.48 & 16/17 & 0.0294 \\
    $\kappa_2(S_-)$        & 2509.05 & 2.40 & 2507.21/2512.23 & 0.0010 \\
    Solve time (ms)        & 214.40  & 1.49 & 212.29/219.98 & 0.0069 \\
    \bottomrule
  \end{tabular}
\end{table}

In this experiment all coefficients of variation are below 3\%, and those of
$|\gamma^-|$, $\kappa_2(S_-)$, and time are below 1\%. \edit{No significant
position dependence was observed over this particular subcell displacement;
the experiment does not establish grid-position independence for arbitrary
geometry or motion.}

\subsection{Fixed-bulk convergence and gradient angle error}
\label{sec:fixed-bulk}

Self-convergence experiments typically exclude a mesh-dependent band of width
$3h$ near boundaries, making comparisons across resolutions less uniform.  We
repeat the smooth-disk self-convergence test using a \emph{fixed} physical
exclusion distance of $0.08$~units from all boundaries. This distance
corresponds to approximately nine cells at $\ell=8$ and is held constant across all four
resolutions.  We additionally report the angular error of the normalized
gradient, which controls the descent direction in the harmonic navigation
application.
\edit{The fixed-region potential, gradient-magnitude, and angular errors are
reported together in \Cref{tab:fixed-bulk}.}

\begin{table}[htbp]
  \centering
  \small
  \caption{Fixed-bulk self-convergence (exclusion distance $0.08$ units,
  approximately nine cells at $\ell=8$). Angular error measures $\arccos(\widehat{\nabla u}_h
  \cdot \widehat{\nabla u}_8)$ between normalized gradients.}
  \label{tab:fixed-bulk}
  \begin{tabular}{lccccc}
    \toprule
    $\ell$ & $h$ & $\|u-u_8\|_\infty$ & $\|\nabla u-\nabla u_8\|_\infty$ & ang.\ error (max) & ang.\ error (mean) \\
    \midrule
    5 & 0.072 & $1.74\times10^{-2}$ & $9.23\times10^{-2}$ & 0.062 & 0.015 \\
    6 & 0.036 & $2.50\times10^{-3}$ & $4.64\times10^{-2}$ & 0.025 & 0.006 \\
    7 & 0.018 & $4.92\times10^{-4}$ & $1.73\times10^{-2}$ & 0.009 & 0.003 \\
    \bottomrule
  \end{tabular}
\end{table}

With a fixed physical exclusion region the potential errors decrease
monotonically; the angular gradient error at $\ell=7$ is below
$0.6^\circ$~($0.009$~rad) in the bulk, confirming that the descent direction
is already well resolved at moderate grid sizes.  The gradient magnitude
error is larger than the directional error, reflecting the well-known
sensitivity of $|\nabla u|$ near curved boundaries.

\subsection{Re-solve cost and component-level profiling}
\label{sec:profiling}
\label{sec:factor-compare}

\edit{We now time the two configurations defined in
\Cref{sec:schur-comparison}. The translation benchmark uses the radius-$0.15$
circle moving from $(0.30,0.15)$ to $(-0.30,-0.15)$, with 21 frames at
$\ell=7$ and seven frames at $\ell=8$; the topology benchmark uses the
13-frame merge--split--merge schedule at $\ell=7$. We benchmark
Algorithm~\ref{alg:full-trace} and Algorithm~\ref{alg:block-update} at every
sampled geometry and rotate their execution order among frames. The steady
averages in \Cref{tab:profiling} exclude the first frame, where
Algorithm~\ref{alg:block-update} initializes and factors $B_{\mathrm{ss}}$.}

\begin{table}[htbp]
  \centering
  \small
  \caption{\edit{Measured steady per-frame times for the two algorithms.
  ``Gain'' is the reduction of Algorithm~\ref{alg:block-update} relative to
  Algorithm~\ref{alg:full-trace}. Times are
  in milliseconds.}}
  \label{tab:profiling}
  \begin{tabular}{lcrrrr}
    \toprule
    Case & Grid & Frames & Algorithm~\ref{alg:full-trace} & Algorithm~\ref{alg:block-update} & Gain \\
    \midrule
    Translation, $\ell=7$ & $127^2$ & 21 & 165.29 & 161.41 & 2.4\% \\
    Translation, $\ell=8$ & $255^2$ & 7  & 392.88 & 371.39 & 5.5\% \\
    Topology change, $\ell=7$ & $127^2$ & 13 & 209.67 & 204.52 & 2.5\% \\
    \bottomrule
  \end{tabular}
\end{table}

\edit{For the translating circle, Algorithm~\ref{alg:block-update} reduces the
steady mean from \SI{165.29}{\milli\second} to
\SI{161.41}{\milli\second} at $\ell=7$ and from
\SI{392.88}{\milli\second} to \SI{371.39}{\milli\second} at $\ell=8$,
corresponding to 2.4\% and 5.5\% gains. The same-grid field differences for
these runs are reported in \Cref{tab:schur-vs-full}. At $\ell=7$,
$n=500$--$502$, $n_{\mathrm s}=456$, and
$n_{\mathrm d}=44$--$46$; at $\ell=8$, the corresponding values are
$1012$--$1013$, $920$, and $92$--$93$. Thus the updated blocks contain only
16.8--17.5\% and 17.4--17.5\% of the full dense entries, respectively. The
static factorization occurs once per run and costs
\SI{0.65}{\milli\second} at $\ell=7$ and \SI{2.74}{\milli\second} at
$\ell=8$. First-frame assembly, which includes $B_{\mathrm{ss}}$, costs
\SI{10.05}{\milli\second} and \SI{37.57}{\milli\second}; the steady block
refresh is substantially smaller.}

\edit{\Cref{tab:block-components} separates that steady Algorithm~\ref{alg:block-update}
time into geometry, block assembly, Schur solve, layer-trace, and reconstruction
components.}

\begin{table}[htbp]
  \centering
  \small
  \caption{\edit{Steady component averages for Algorithm~\ref{alg:block-update},
  excluding first-frame initialization. Times are in milliseconds.}}
  \label{tab:block-components}
  \resizebox{\linewidth}{!}{%
  \begin{tabular}{lrrrrrr}
    \toprule
    Case & Geometry & Block assembly & Schur solve & Layer trace & DST & Geometry share \\
    \midrule
    Translation, $\ell=7$ & 153.28 & 1.30 & 0.36 & 3.11 & 1.01 & 95.0\% \\
    Translation, $\ell=8$ & 342.84 & 4.68 & 1.43 & 14.19 & 3.51 & 92.3\% \\
    Topology change, $\ell=7$ & 194.37 & 3.30 & 0.58 & 4.21 & 1.05 & 95.0\% \\
    \bottomrule
  \end{tabular}
  }
\end{table}

\edit{The component breakdown in \Cref{tab:block-components} shows why the
algebraic reduction produces only a modest end-to-end gain. In the
topology-change test, $n$ varies from 524 to 592 while
$n_{\mathrm s}=456$ and $n_{\mathrm d}=68$--$136$.
Algorithm~\ref{alg:block-update} remains faster than
Algorithm~\ref{alg:full-trace} even though the refreshed fraction
rises to 40.7\%. The block update is therefore active and measurable, but its
end-to-end benefit is modest because rebuilding the boundary classification,
intersections, and local collocation data consumes 93--95\% of its total time.
The code intentionally does not refactor the static block at each frame. The
next performance target is a swept-band geometry update; kernel compression is
also needed when the boundary size becomes much larger.}



\begin{remark}

All experiments use the LGF-BAE solver with a second-order five-point
Laplacian, the single-layer potential, and Dirichlet boundary conditions.
The physical domain is $[-1,1]^2$ and the auxiliary finite box adds a padding
length $l=0.15$, giving the grid extent $[-1.15,1.15]^2$. \edit{We denote the
refinement level by $\ell$ and the number of interior grid points per coordinate
by $N_g=2^\ell-1$. Thus $h=2.30/2^\ell$, giving
$h\approx0.072,0.036,0.018,0.009$ for $\ell=5,6,7,8$ on grids of size
$31^2,63^2,127^2,255^2$. The Dirichlet sine transform is implemented by FFTs
of length $2(N_g+1)$ (256 at $\ell=7$); this transform length is not a separate
free-space padding box.} The LGF kernel is tabulated once using near-field
quadrature and the far-field expansion in \cref{eq:lgf-asymp}, with the code
switching at radius 30 grid points. The recorded environment is an
Apple-silicon MacBook Pro running Python~3.14.3, NumPy~2.4.6, and
SciPy~1.17.1. \edit{The recorded machine is a MacBook Pro Mac16,8 with an
Apple M4 Pro processor (14 CPU cores) and 48~GB memory.}

\end{remark}

\section{Application: harmonic path planning}
\label{sec:application}

The manufactured-solution tests show near-second-order bulk convergence of
both $u_h$ and $\grad_h u$, and Algorithms~\ref{alg:full-trace} and
\ref{alg:block-update} agree to numerical
precision. We next use path planning as an
application. The experiments below show what a reliable, gradient-accurate
repeated elliptic solver enables, not novelty in harmonic navigation itself.
The planner is not proposed as a state-of-the-art robotics planner. It is
used as a gradient-sensitive application demonstrating repeated elliptic solves
on changing implicit domains.

We demonstrate the method on static and dynamic path-planning benchmarks.
\edit{The static single-path and multi-query scripts compute the field with
Algorithm~\ref{alg:full-trace}, whereas the moving-geometry scripts compute the
changing fields with Algorithm~\ref{alg:block-update}. Both use the normalized
NAG path update described below.} Path extraction is a separate post-processing
step.

\subsection{Path extraction method}
\label{sec:paths}

We use normalized NAG descent as a low-cost path-extraction method:
 \begin{equation}
\mathbf{y}_k = \xx_k + \beta(\xx_k-\xx_{k-1}),\qquad
\xx_{k+1} = \xx_k - \eta\,\frac{\grad u(\mathbf{y}_k)}{\norm{\grad u(\mathbf{y}_k)}},
\end{equation}
with momentum $\beta=0.9$, requiring one gradient evaluation per step. The
static tests use $\eta=h/4$; the moving tests use the fixed step sizes stated
with each experiment because they take 15 agent steps per geometry frame.
The static demonstrations apply an additional mask-based endpoint correction
when a proposed step enters an obstacle. In every dynamic driver, momentum is
retained when the field changes: the gradient is sampled at the look-ahead
point $\mathbf y_k$, but the accepted displacement starts from $\xx_k$. The
moving-geometry implementation performs neither segment-level collision
detection nor endpoint projection or step rejection.

\subsection{Single-path planning on static benchmarks}
\label{sec:static-paths}

\Cref{fig:static-paths} shows single NAG-descent paths on the two
representative geometries introduced in \Cref{sec:self-convergence}.
The agent follows $-\grad u/\|\grad u\|$ from a specified start
to the goal (green).  Both paths arrive successfully; the path length
reflects the harmonic field's tendency to take smooth, obstacle-avoiding
routes rather than distance-optimal shortcuts.  On the cross-shaped obstacle
the path curves around the obstacle corners without being explicitly
programmed to do so---the Dirichlet condition $u=1$ on the obstacle
boundary creates a repulsive gradient that naturally steers the agent
along the obstacle's convex hull.

\begin{figure}[htbp]
\centering
\subcaptionbox{Smooth disk ($L=1.51$)}
  {\includegraphics[width=0.35\linewidth]{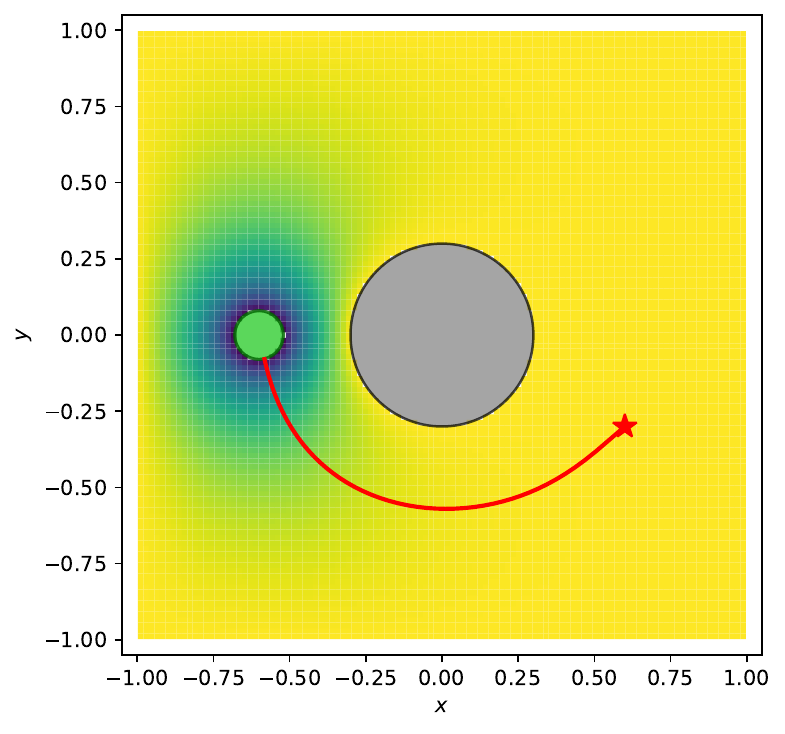}}
  ~
\subcaptionbox{Cross-shaped obstacle ($L=2.05$)}
  {\includegraphics[width=0.35\linewidth]{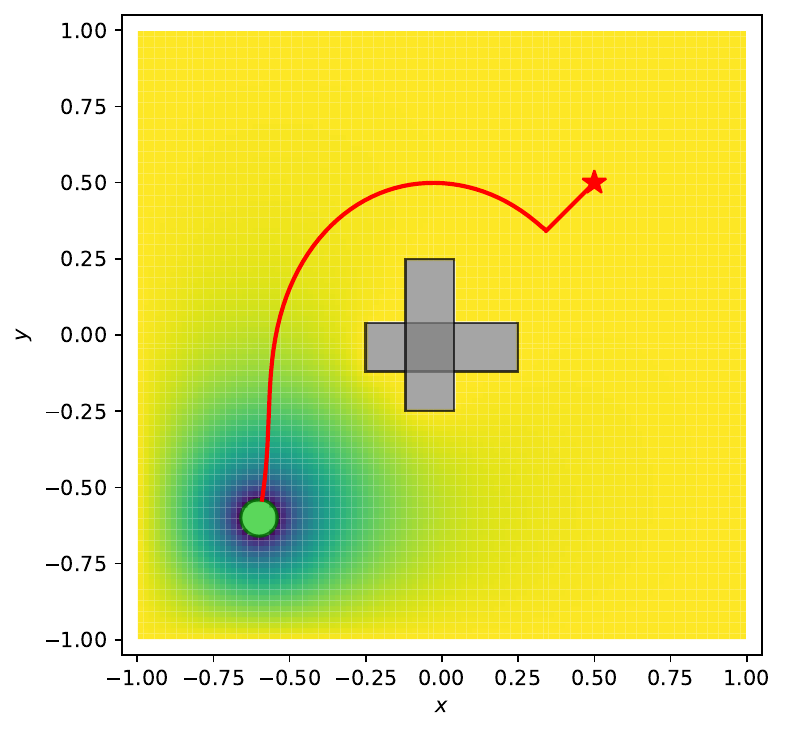}}
\caption{Single-path planning on the smooth disk and cross-shaped
obstacle geometries ($\ell=7$, NAG descent).  Both paths arrive; the
cross-shaped obstacle induces a natural avoidance maneuver around its corners.}
\label{fig:static-paths}
\end{figure}

A single field solve serves unlimited start positions. \Cref{fig:multiquery}
shows 50 paths from randomly scattered starts converging to a single goal
after one \SI{3.5}{\second} solve. The displayed field cost therefore amortizes
to approximately \SI{71}{\milli\second} per path at 50 queries. The reported
per-path extraction cost is \SI{22}{\milli\second}, giving a total of
approximately \SI{93}{\milli\second} per path; only the field component
continues to decrease as the query count grows.

\begin{figure}[htbp]
\centering
\includegraphics[width=0.35\linewidth]{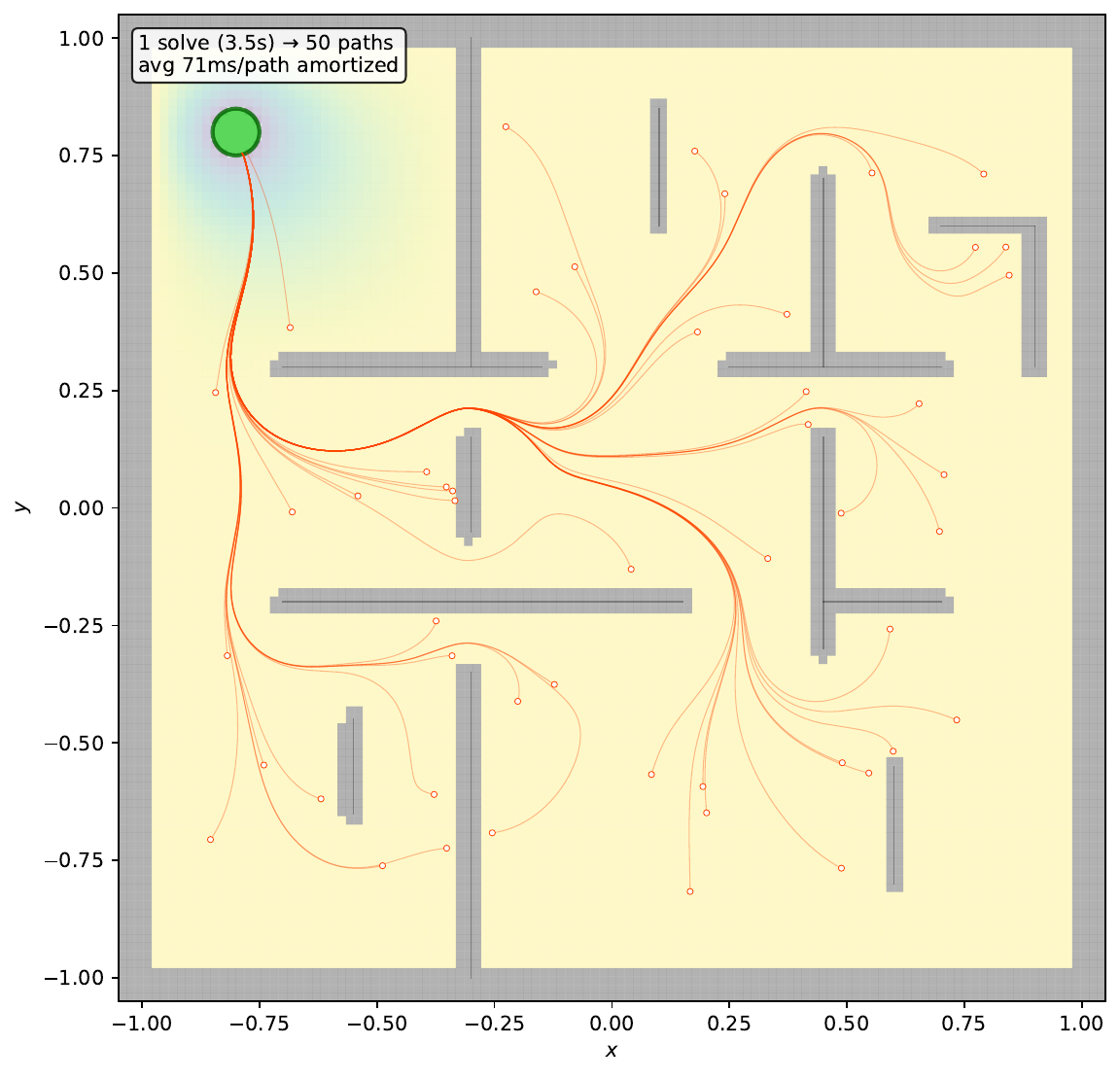}
\caption{One field solve serves 50 paths from random starts. The displayed
\SI{3.5}{\second} field solve amortizes to approximately
\SI{71}{\milli\second} per path; including the reported
\SI{22}{\milli\second} extraction cost gives approximately
\SI{93}{\milli\second} per path.}
\label{fig:multiquery}
\end{figure}


\FloatBarrier
\subsection{Dynamic replanning with moving obstacles}
\label{sec:app-dynamic}

\rev{For the single-moving-obstacle experiment,} the start and goal are
$(0.75,-0.75)$ and $(-0.80,0.80)$, respectively. An
obstacle drifts diagonally from $(0.5,0.5)$ to $(-0.5,-0.5)$ over 40
obstacle-motion frames; at each frame the field is re-solved with
Algorithm~\ref{alg:block-update}
(\Cref{sec:method-dynamic}) and the agent takes 15 normalized NAG steps with
$\beta=0.9$. Three step sizes, $\eta=0.0020$, $0.0025$, and $0.0040$, are
interpreted as slow, medium, and fast agent speeds, respectively. If an
agent has not reached the goal by the end of the
obstacle-motion schedule, it continues on the final field; this is why some
snapshot indices exceed 40.

\Cref{fig:single-moving-obstacle} shows the agent trajectories with one
obstacle. The Dirichlet value $u=1$ on the obstacle boundary creates a repulsive
gradient, and the medium-speed trajectory has the longest reported path. All
three agents reach the goal with path lengths $2.40$, $2.92$, and $2.32$ for
the slow, medium, and fast cases, respectively. The corresponding step counts
are 1202, 1166, and 579. The 40 distinct obstacle fields are computed once and
reused across the three speed trials. \edit{The temporal progression of each
speed case is resolved in \Cref{fig:single-moving-obstacle-keyframes}. Because the script
does not test segment--obstacle intersections, these runs support a qualitative
avoidance demonstration but not a verified zero-collision claim.}

\paragraph{Algorithm used and verification}
\edit{The overview and snapshot figures in this subsection were regenerated
with Algorithm~\ref{alg:block-update}. The swept circular obstacle is enclosed
by a capsule, the invariant $B_{\mathrm{ss}}$ block is factored once, and the
maximum block residual over the 40 distinct fields is
$6.44\times10^{-15}$. The representative same-grid comparison in
\Cref{sec:schur-comparison} shows that these block-update fields agree with
Algorithm~\ref{alg:full-trace} to about $10^{-9}$.}

\begin{figure}[htbp]
\centering
\subcaptionbox{Slow, snapshot 1}
  {\includegraphics[width=0.29\linewidth,trim=0 0 0 24,clip]{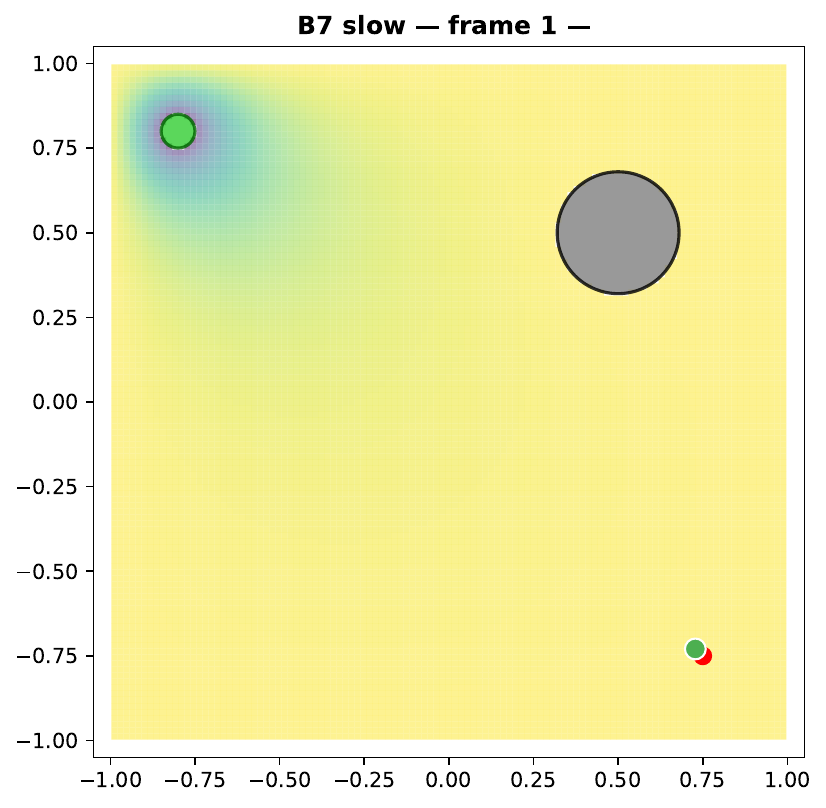}}
\subcaptionbox{Slow, snapshot 21}
  {\includegraphics[width=0.29\linewidth,trim=0 0 0 24,clip]{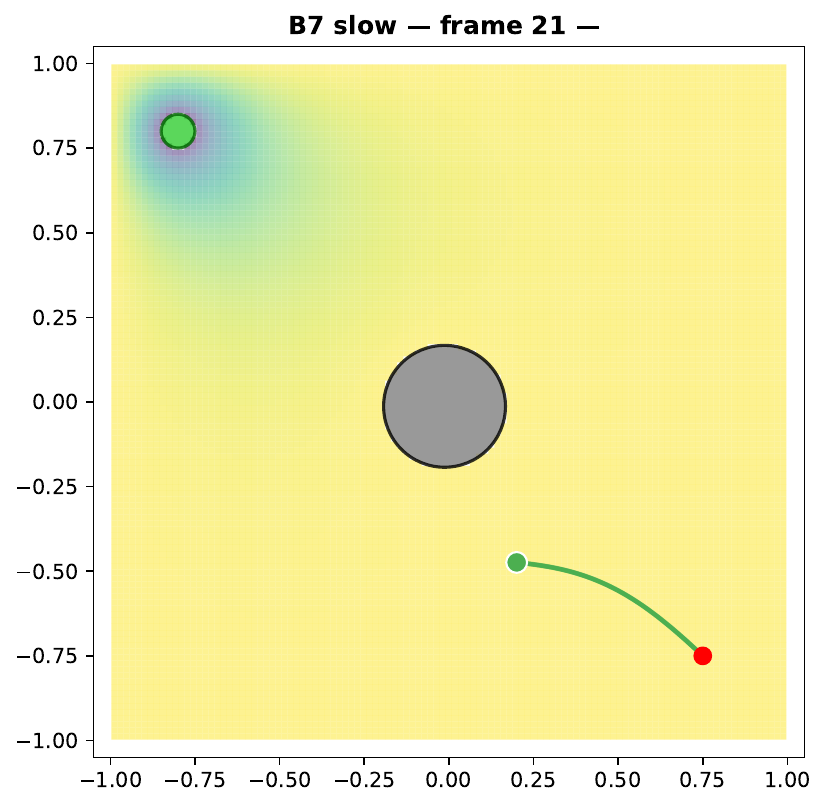}}
\subcaptionbox{Slow, snapshot 47}
  {\includegraphics[width=0.29\linewidth,trim=0 0 0 24,clip]{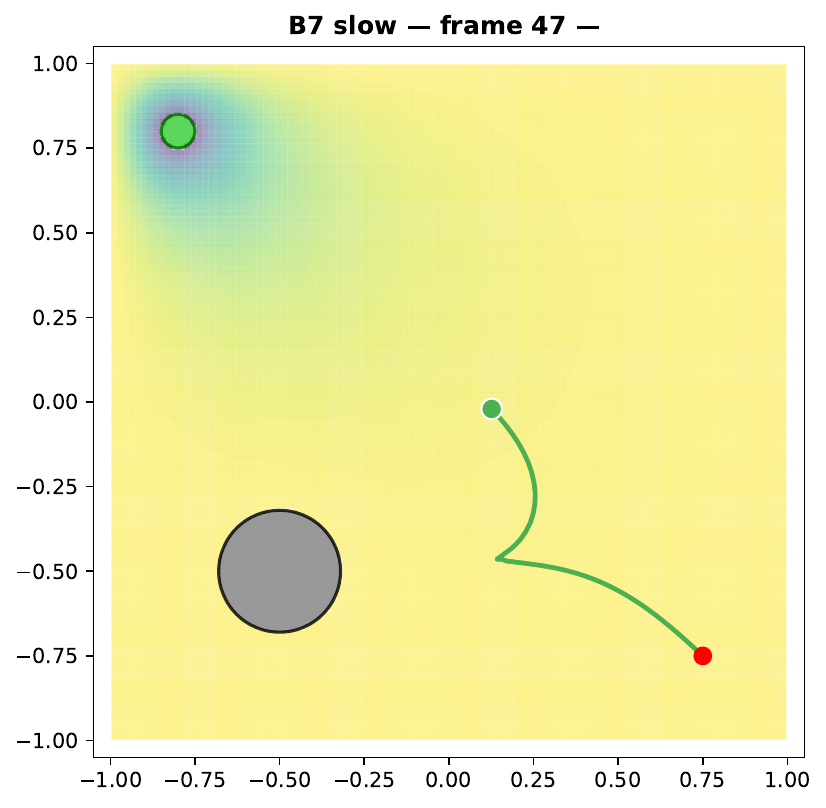}}\\[6pt]
\subcaptionbox{Medium, snapshot 1}
  {\includegraphics[width=0.29\linewidth,trim=0 0 0 24,clip]{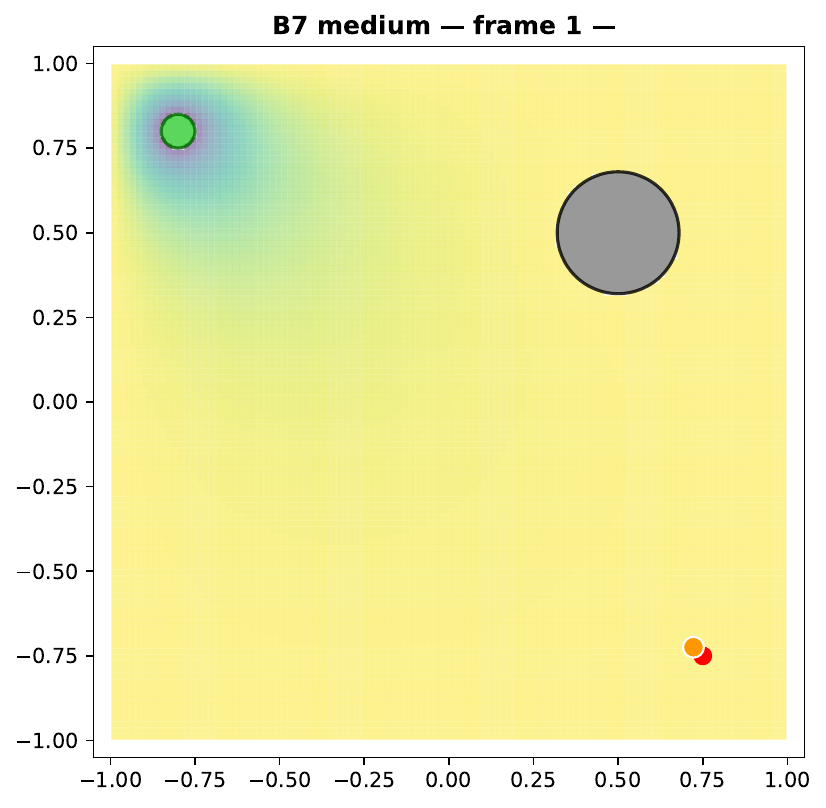}}
\subcaptionbox{Medium, snapshot 21}
  {\includegraphics[width=0.29\linewidth,trim=0 0 0 24,clip]{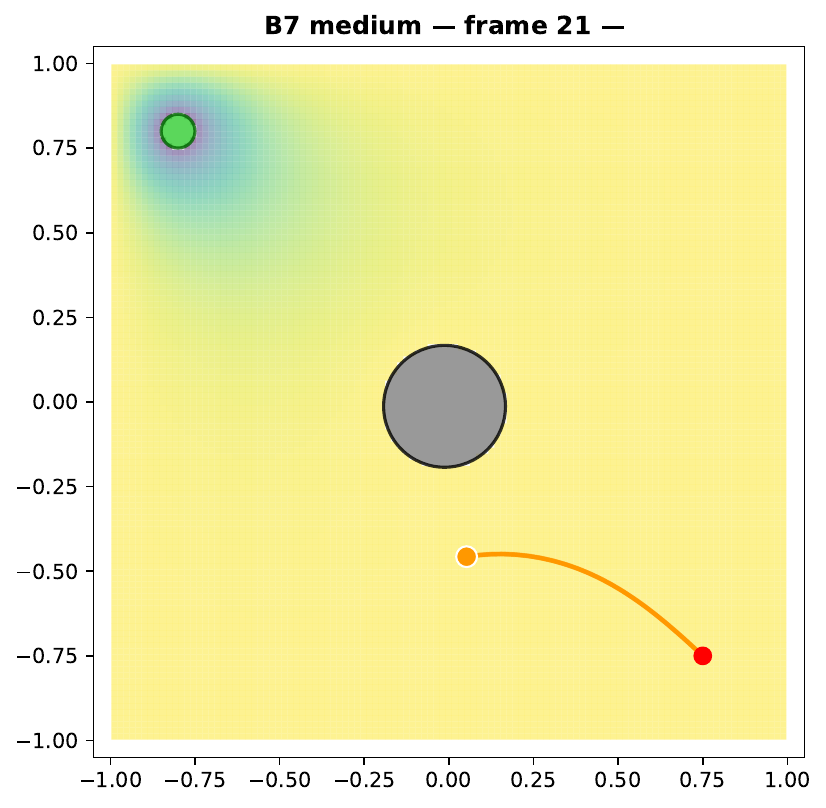}}
\subcaptionbox{Medium, snapshot 46 (arrival)}
  {\includegraphics[width=0.29\linewidth,trim=0 0 0 24,clip]{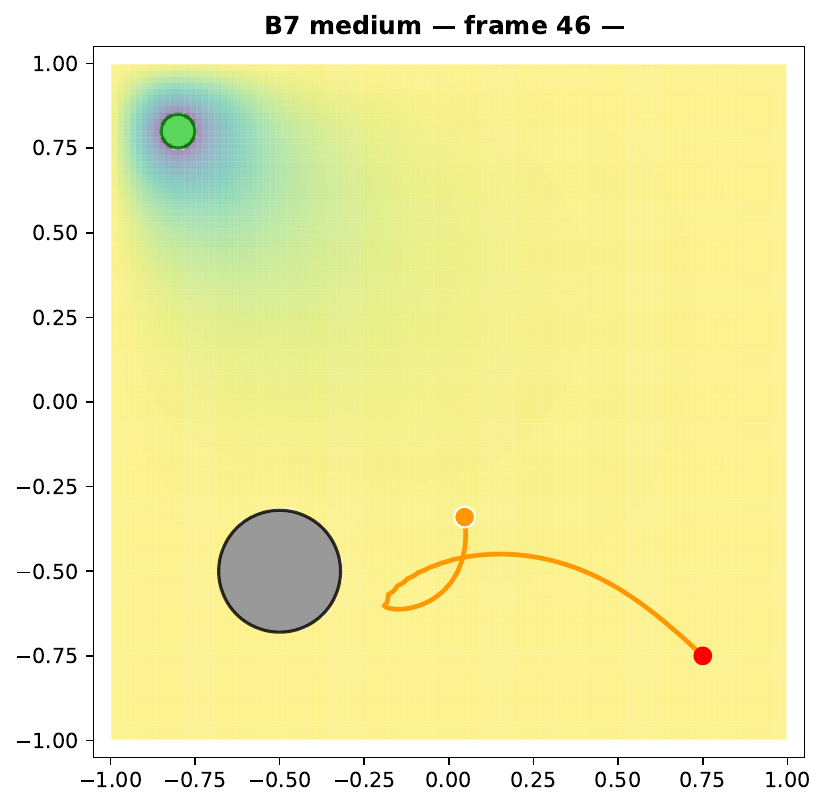}}\\[6pt]
\subcaptionbox{Fast, snapshot 1}
  {\includegraphics[width=0.29\linewidth,trim=0 0 0 24,clip]{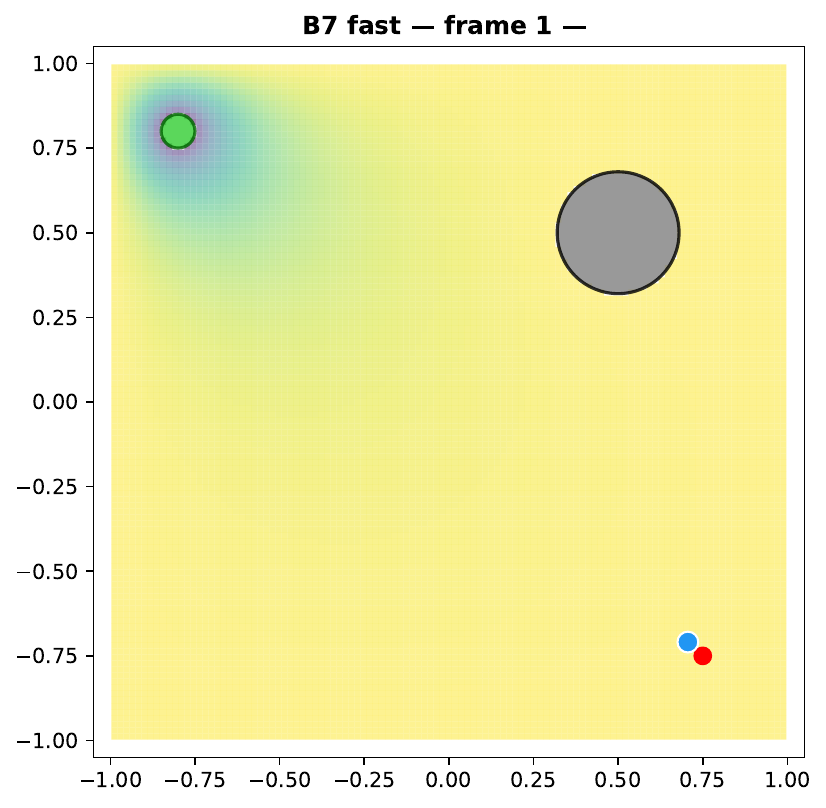}}
\subcaptionbox{Fast, snapshot 21}
  {\includegraphics[width=0.29\linewidth,trim=0 0 0 24,clip]{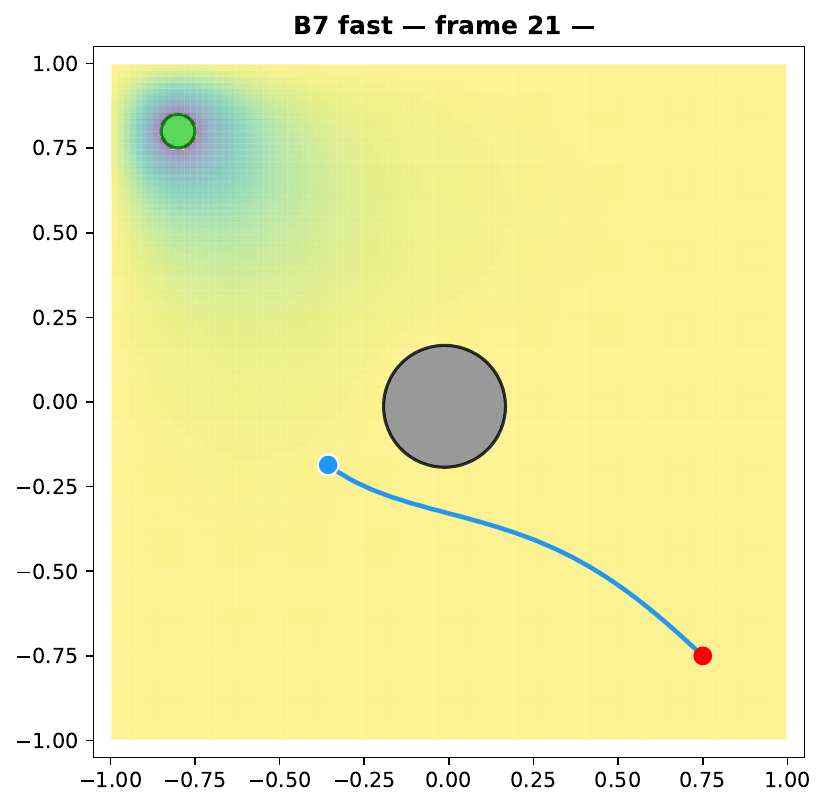}}
\subcaptionbox{Fast, snapshot 39 (arrival)}
  {\includegraphics[width=0.29\linewidth,trim=0 0 0 24,clip]{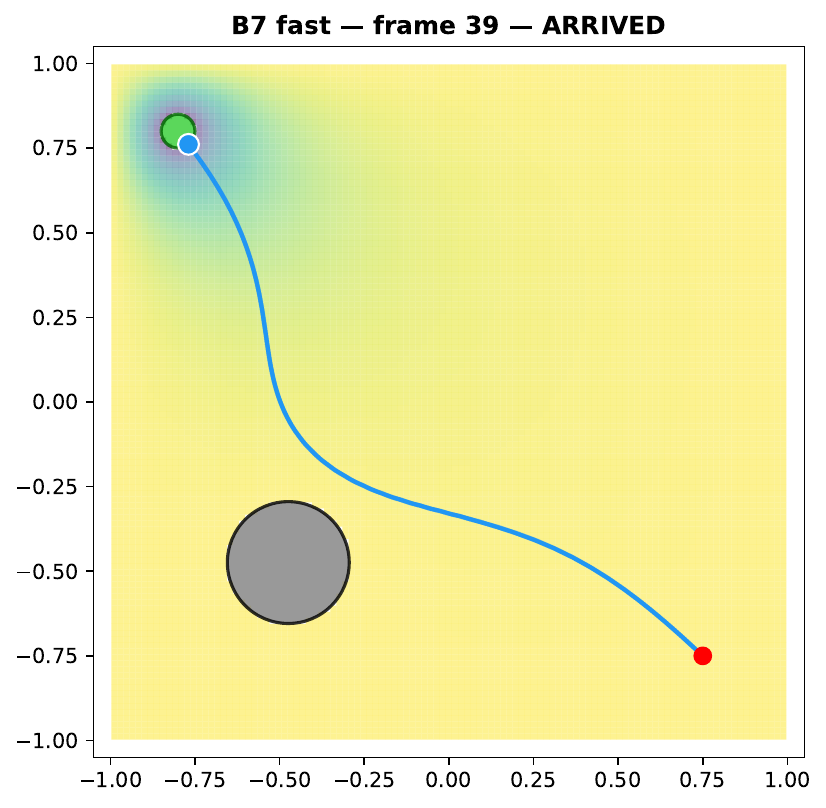}}
\caption{\edit{Algorithm~\ref{alg:block-update} with normalized NAG: selected
key frames for the single moving-obstacle experiment.
Slow (top): continued motion on the final field. Medium (middle): the longest
reported detour. Fast (bottom): arrival during the obstacle-motion schedule.}}
\label{fig:single-moving-obstacle-keyframes}
\end{figure}

\begin{figure}[htbp]
\centering
\includegraphics[width=0.35\linewidth]{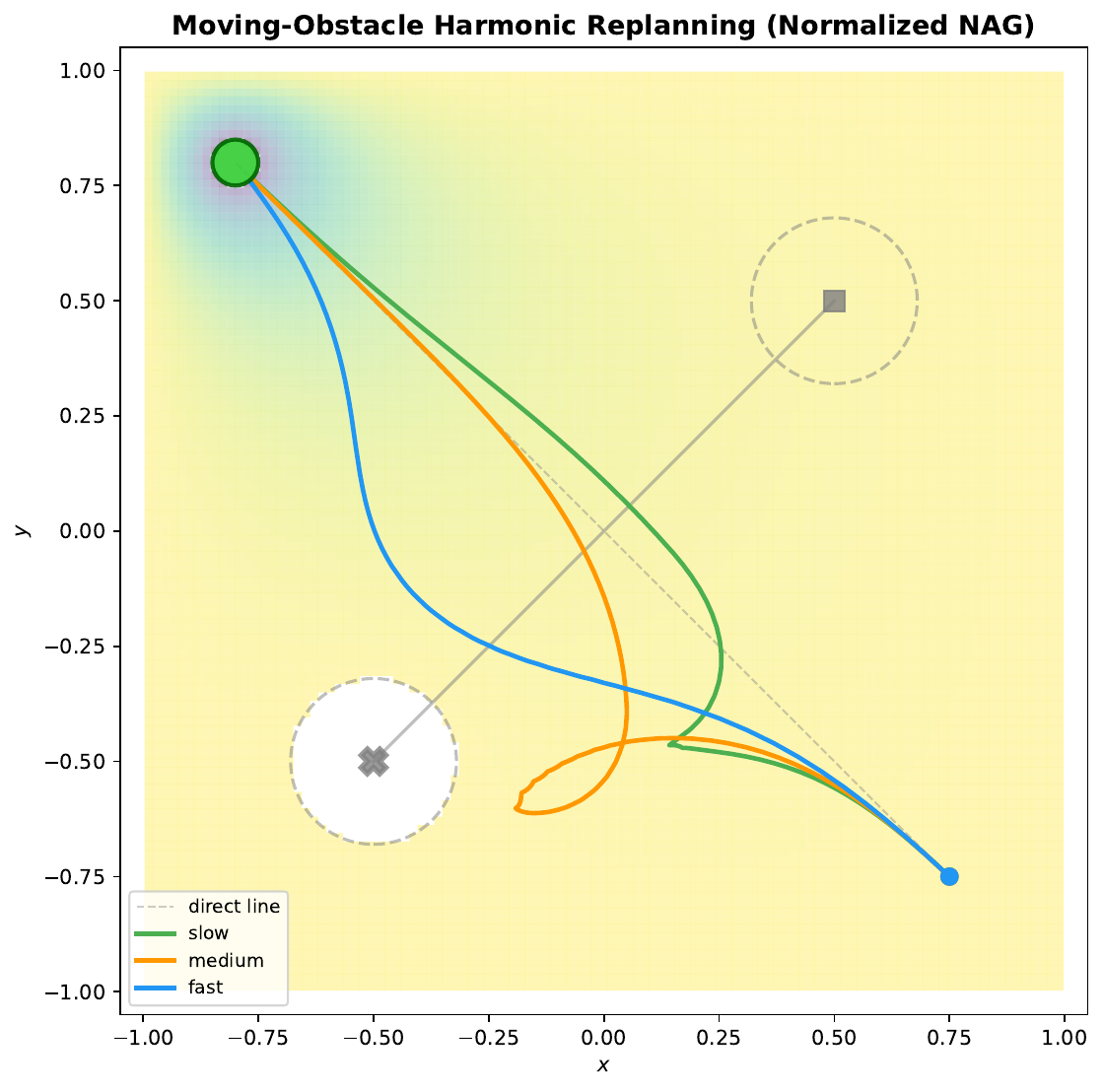}
\caption{\edit{Generated with Algorithm~\ref{alg:block-update} and normalized
NAG ($\beta=0.9$).} One moving obstacle and three step sizes. The dashed gray line is
the direct path without the obstacle. The medium-speed trajectory (orange) has
the longest reported path ($2.92$ vs.\ $2.32$--$2.40$). Collision-free motion
is not programmatically verified.}
\label{fig:single-moving-obstacle}
\end{figure}

\rev{A second experiment uses two obstacles on crossing paths.
\Cref{fig:two-moving-obstacles,fig:two-moving-obstacles-keyframes} show the
aggregate trajectory and selected frames, respectively. \edit{These fields are
computed with Algorithm~\ref{alg:block-update}; the static block is factored
once and the maximum block residual is $6.22\times10^{-15}$. With
$\eta=0.0025$, the normalized-NAG path reaches the goal in 1222 steps with
length $3.05$.}}
\begin{figure}[htbp]
\centering
\includegraphics[width=0.35\linewidth]{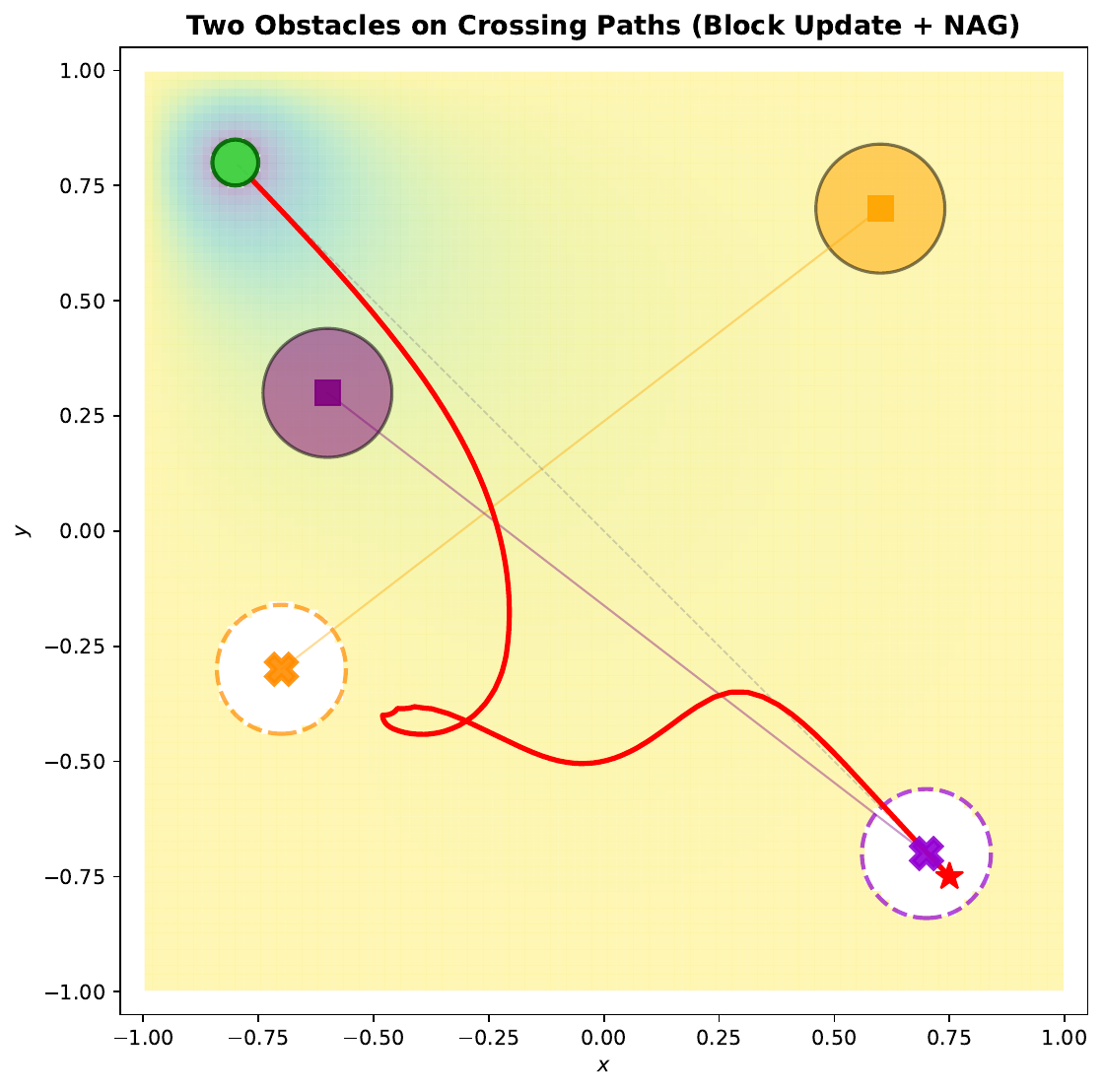}
\caption{\edit{Generated with Algorithm~\ref{alg:block-update} and normalized
NAG.} Two obstacles on crossing paths. The sampled agent trajectory (red)
passes through the changing gap. The script contains no segment-level collision
test, so the figure is qualitative.}
\label{fig:two-moving-obstacles}
\end{figure}

\begin{figure}[htbp]
\centering
\subcaptionbox{Frame 1}
  {\includegraphics[width=0.31\linewidth]{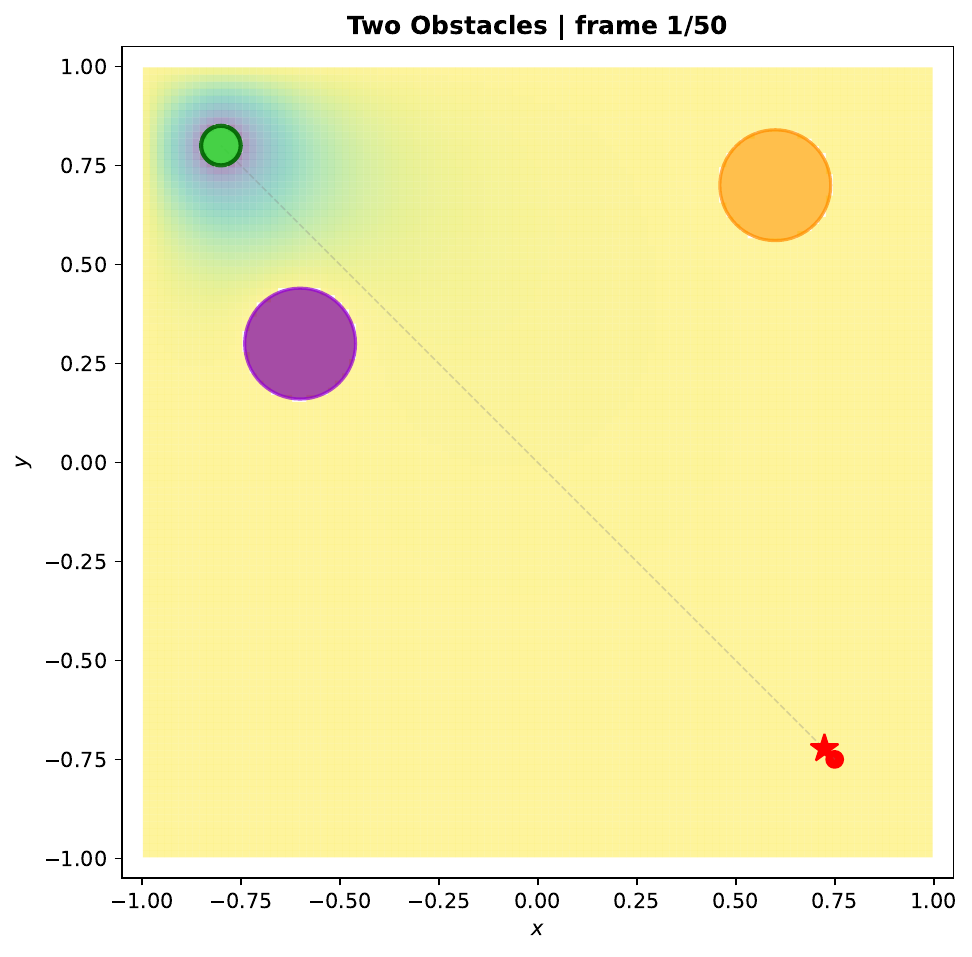}}
\subcaptionbox{Frame 21}
  {\includegraphics[width=0.31\linewidth]{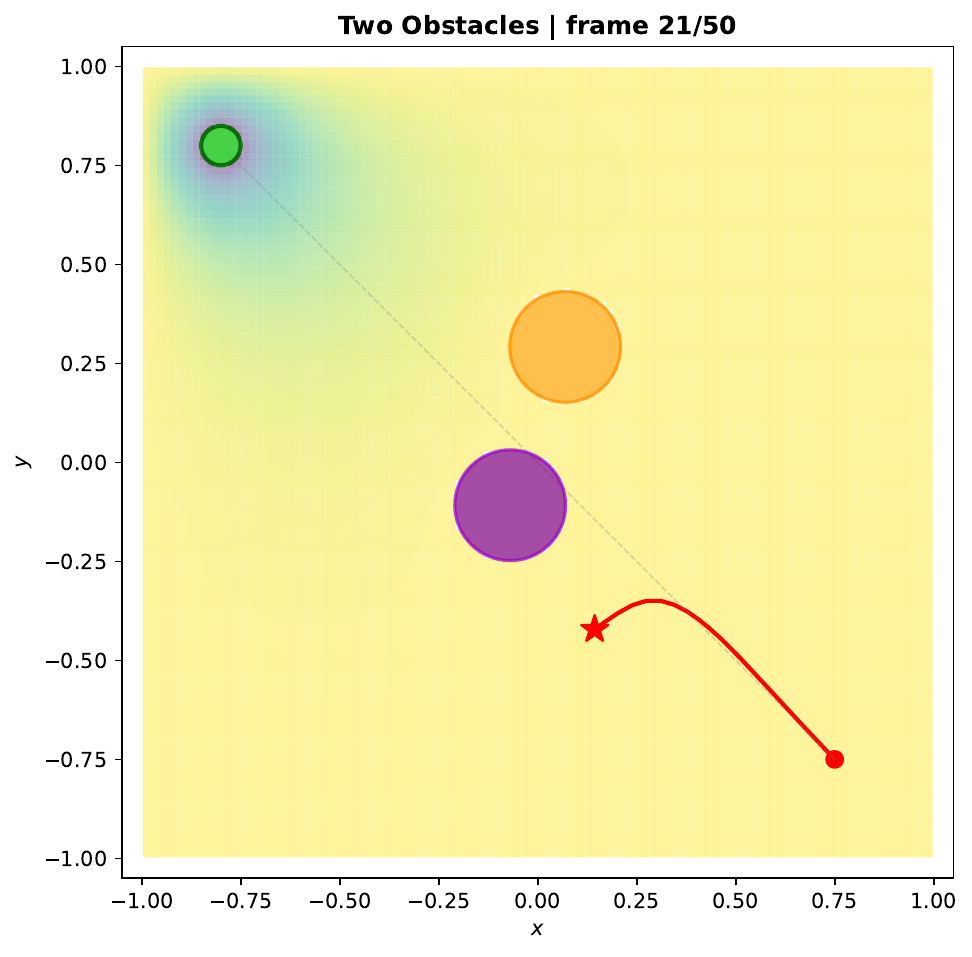}}
\subcaptionbox{Frame 41}
  {\includegraphics[width=0.31\linewidth]{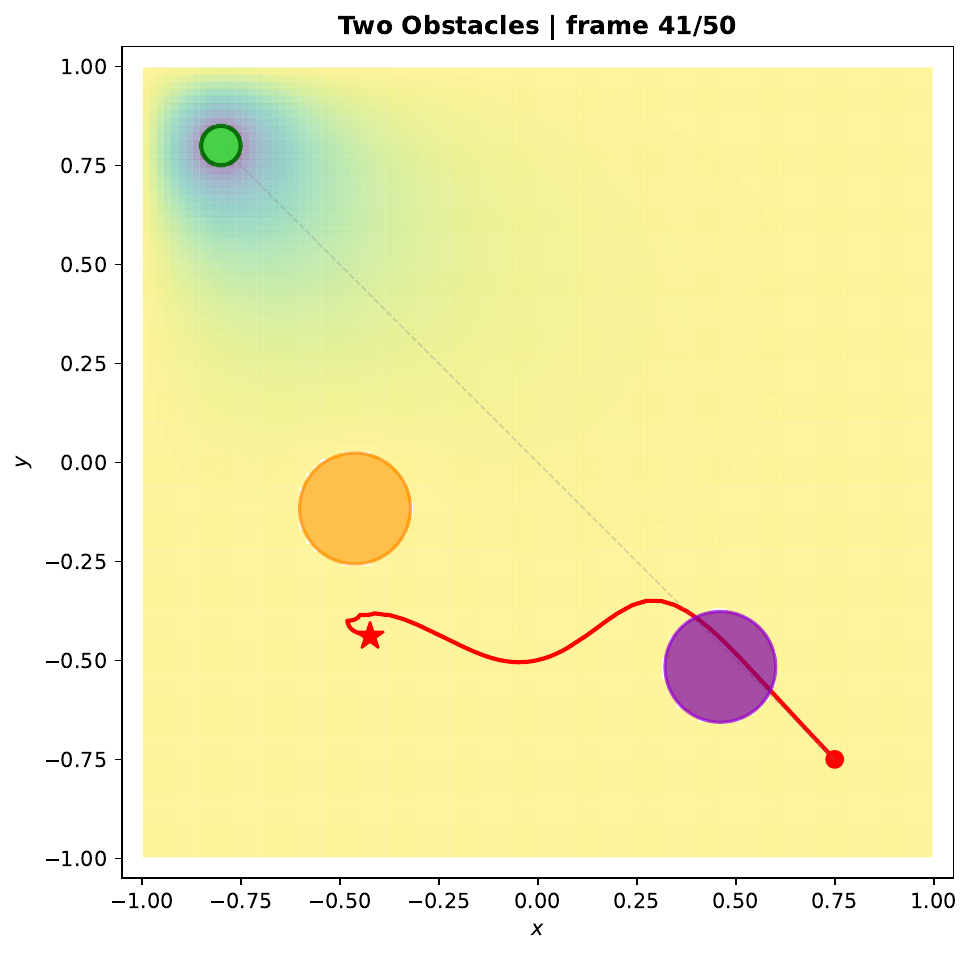}}
\caption{\edit{Algorithm~\ref{alg:block-update} with normalized NAG:
two-obstacle crossing experiment at frames 1, 21, and 41 of the
50-frame schedule.}}
\label{fig:two-moving-obstacles-keyframes}
\end{figure}

\FloatBarrier

\subsection{Deforming geometry and topology change}
\label{sec:topology-change}

The static/dynamic decomposition of \Cref{sec:method-dynamic} does not assume
rigid motion; it only requires that the boundary is defined by a level set.
\Cref{fig:topology-change-keyframes} tests this by having a single obstacle undergo a continuous
shape deformation that includes a \emph{topology change}: two overlapping
circles (one connected component) drift apart until they separate into two
disconnected obstacles, then merge back. The grid, LGF kernel, and sine-transform solver
are unchanged throughout.

\edit{The displayed trajectory frames were regenerated with
Algorithm~\ref{alg:block-update} on a 36-frame radius-$0.15$ schedule.
Endpoint-excluded separation ramps cross from merged to split without sampling
the singular exact-tangency configuration. The static block is factored once
and the maximum block residual is $6.11\times10^{-15}$. With $\eta=0.0040$,
normalized NAG reaches the goal in 624 steps with path length $2.50$. The separate 13-frame
radius-$0.22$ comparison in \Cref{sec:schur-comparison} uses
Algorithms~\ref{alg:full-trace} and \ref{alg:block-update}:
the dynamic count varies from 68 to 136 of 524--592 unknowns, and the
Algorithm~\ref{alg:block-update} field differs from
Algorithm~\ref{alg:full-trace} by at most
$1.03\times10^{-9}$.}

\begin{figure}[htbp]
\centering
\subcaptionbox{Frame 4 (merged)}
  {\includegraphics[width=0.31\linewidth]{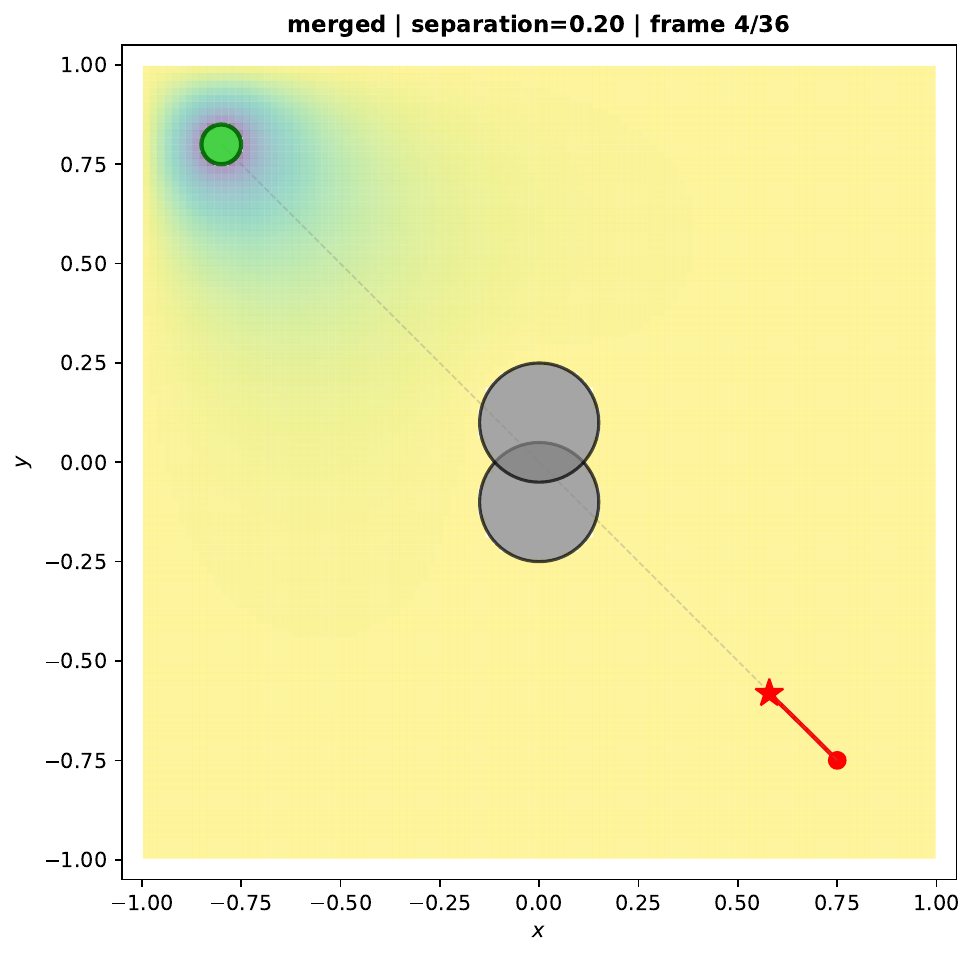}}
\subcaptionbox{Frame 19 (split)}
  {\includegraphics[width=0.31\linewidth]{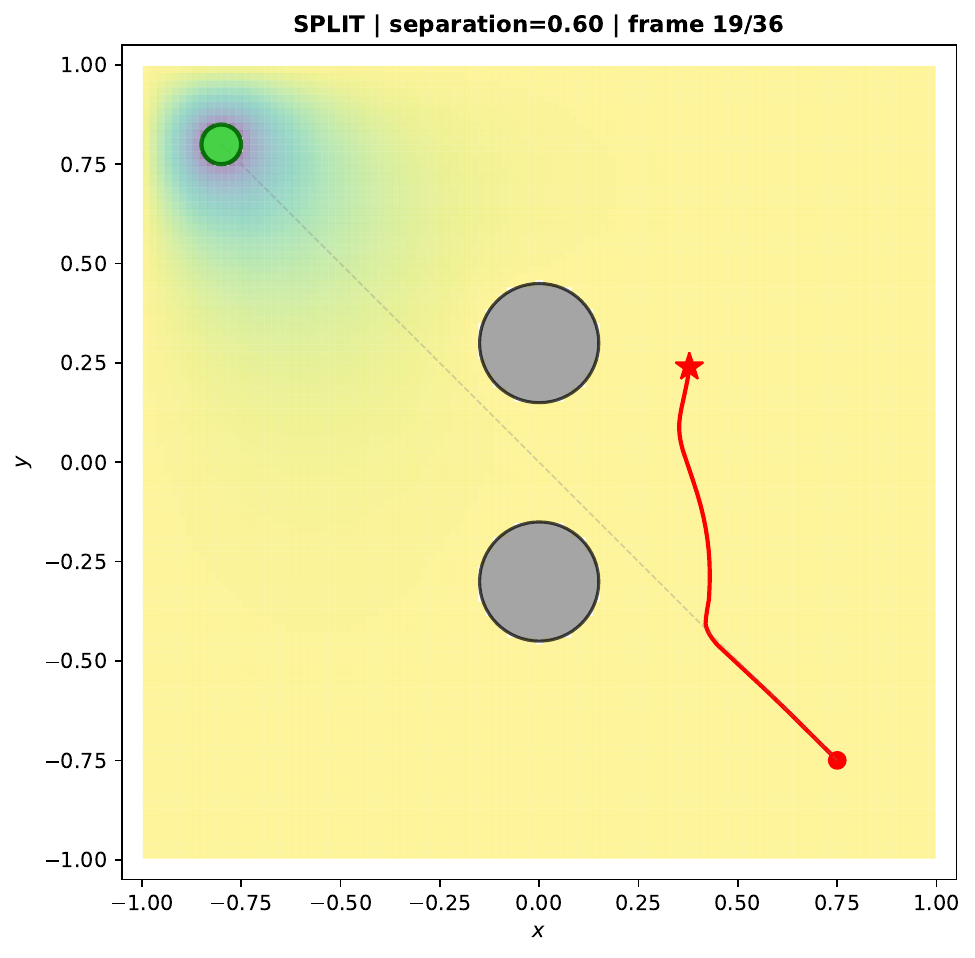}}
\subcaptionbox{Frame 25 (merging)}
  {\includegraphics[width=0.31\linewidth]{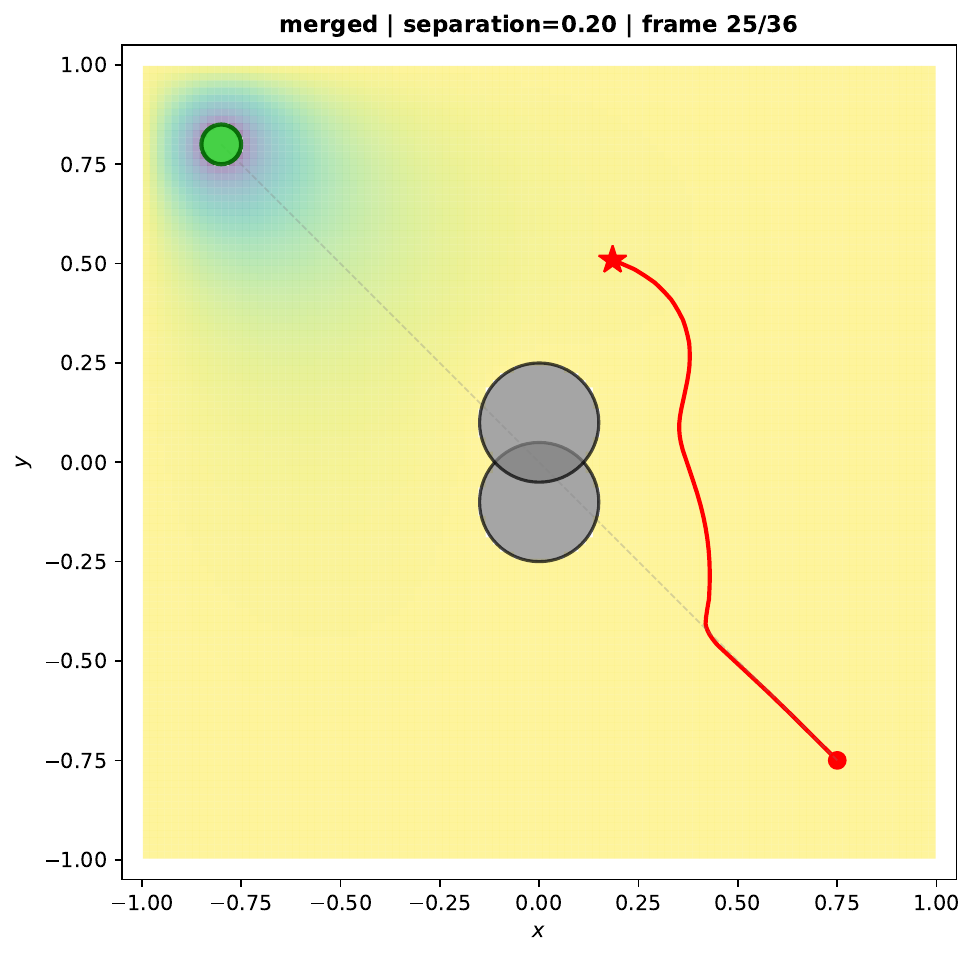}}
\caption{\rev{Algorithm~\ref{alg:block-update} with normalized NAG:
topology-change experiment: merged $\to$ split
$\to$ merging.}}
\label{fig:topology-change-keyframes}
\end{figure}

\FloatBarrier

\subsection{Appearing obstacle}
\label{sec:appearing-obstacle}

A more dramatic test of the method's reactivity is an obstacle that
\emph{appears suddenly} in the agent's path, with no prior warning.
\Cref{fig:appearing-obstacle-keyframes} shows the scenario: the agent follows the free-space harmonic
field for 12 frames, at which point a circular obstacle ($r=0.18$)
materializes at the origin. The field is re-solved once with the new level
\edit{set by Algorithm~\ref{alg:block-update}; no remeshing is needed. The
same $B_{\mathrm{ss}}$ factorization is reused before and after appearance, and
the two solves have a maximum block residual of $5.77\times10^{-15}$.} The sampled trajectory
transitions from the free-space path (dashed green) to the replanned
path (solid red), turns around the newly appeared obstacle, and reaches the
goal under normalized NAG with $\eta=0.0040$ in 680 steps (path length 2.72
vs.\ 2.14 direct).

\begin{figure}[htbp]
\centering
\subcaptionbox{Frame 1 (free space)}
  {\includegraphics[width=0.31\linewidth]{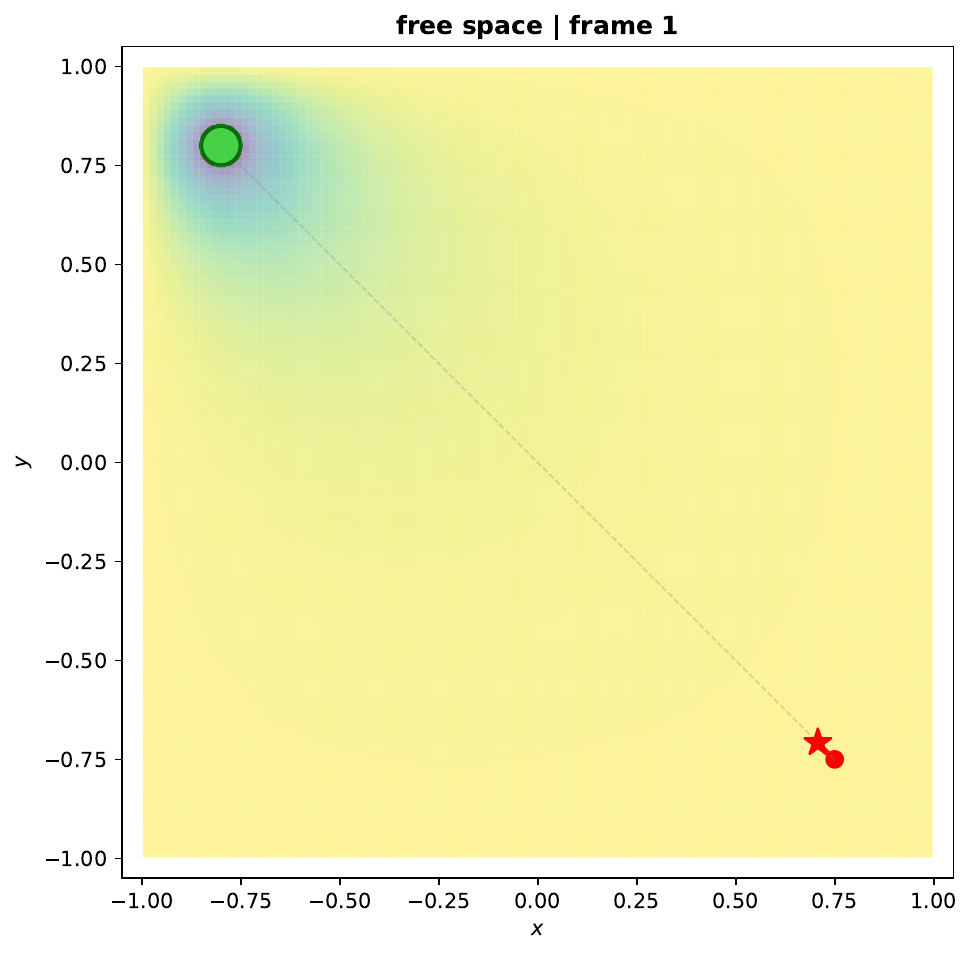}}
\subcaptionbox{Frame 12 (obstacle appears)}
  {\includegraphics[width=0.31\linewidth]{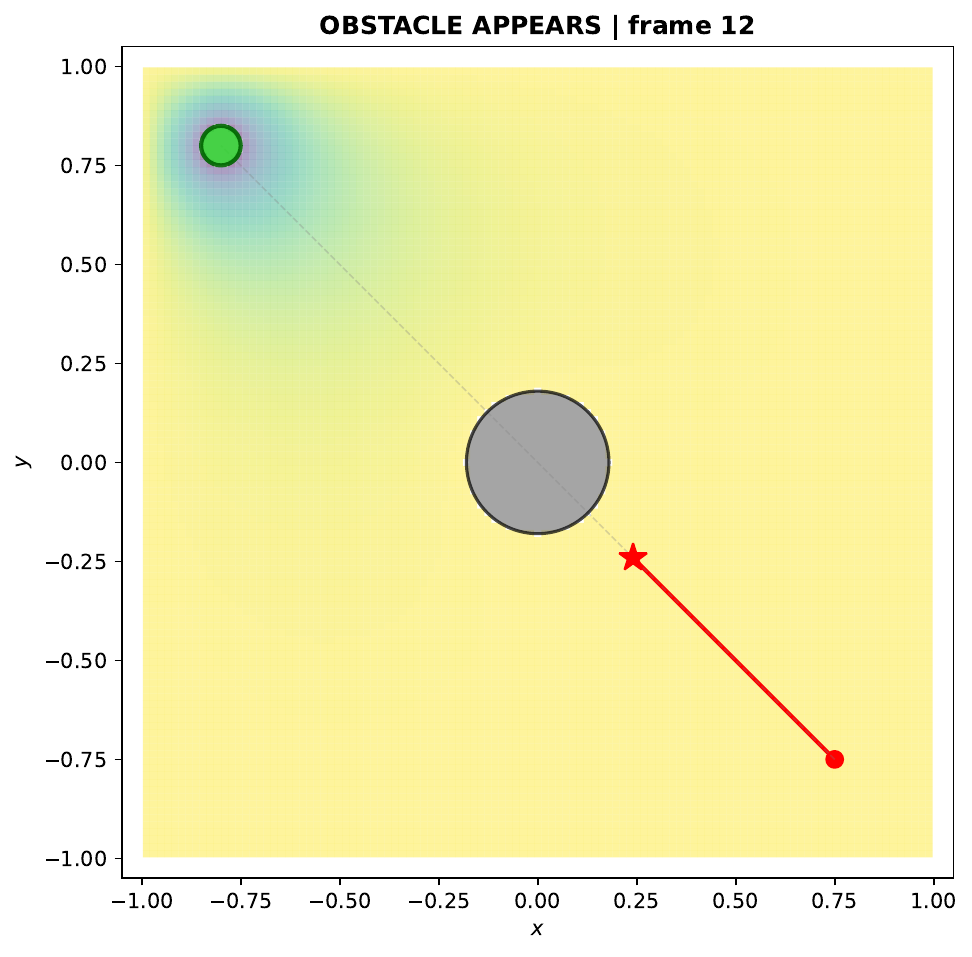}}
\subcaptionbox{Frame 46 (arrival)}
  {\includegraphics[width=0.31\linewidth]{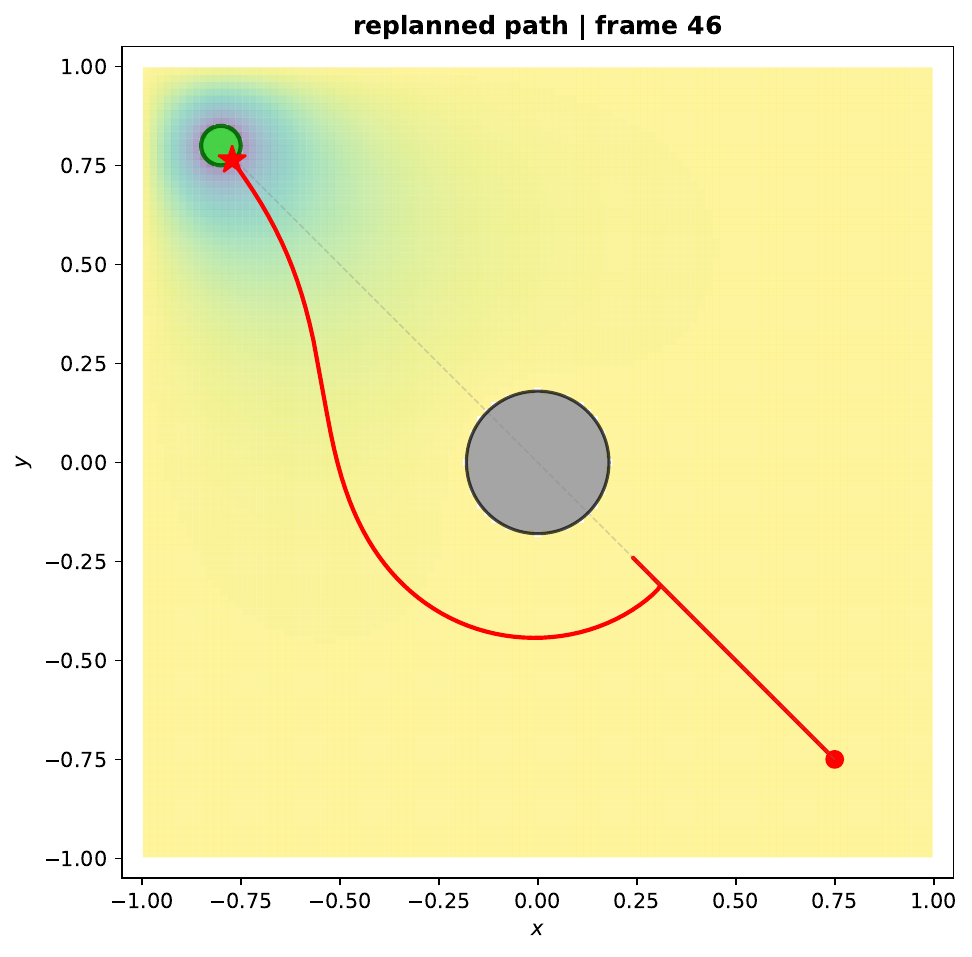}}
\caption{\rev{Algorithm~\ref{alg:block-update} with normalized NAG:
appearing-obstacle experiment. The obstacle
appears at frame 12, and the agent replans without remeshing.}}
\label{fig:appearing-obstacle-keyframes}
\end{figure}

\FloatBarrier
\section{Discussion and extensions}
\label{sec:discussion}

\subsection{Relation to fast marching and sampling-based planners}
\label{sec:comparison}
Fast marching and Eikonal solvers~\cite{sethian1996fast,kimmel1998computing} are the natural
competitors for single-query
workspace path planning. They are usually preferable when the desired object is
a shortest-distance field and the main requirement is minimal solve cost. The
harmonic field used here solves a different problem: it produces a smooth
potential governed by an elliptic boundary value problem, with obstacle
repulsion encoded directly through Dirichlet data. This is not a claim that a
Laplace solve beats fast marching for a one-shot shortest-path query;
fast-marching planners~\cite{garrido2006path} are natural PDE-based alternatives
for that setting. The case for the present method is strongest when the geometry
is implicit, when many queries reuse the same computed field, and when the
elliptic potential and its gradient are the desired outputs rather than a
distance-optimal path.

Sampling-based planners~\cite{lavalle2006planning,karaman2011sampling,gammell2015batch} address a
different regime. They remain the appropriate
choice in high-dimensional configuration spaces or when kinodynamic constraints
are central. Among other field-based constructions, analytic navigation
functions~\rev{\cite{rimon1990exact}} avoid a PDE solve at the cost of star-shaped geometry
assumptions, and continuum / dynamic-potential approaches have been used for
large-scale crowd simulation~\cite{treuille2006continuum}; the present method instead
solves the elliptic problem directly on arbitrary implicit geometry.
\edit{For moving geometry, the present work implements density-block reuse
and demonstrates a 2.4--5.5\% steady-frame reduction relative to the full trace
baseline in the translating-circle tests and a 2.5\% reduction through a
merge--split--merge topology change with up to 23\% dynamic unknowns. The
modest end-to-end gain reflects the remaining global geometry update, not
repeated factorization of the static block.}

\subsection{Where the method is the right tool}
\label{sec:where-right}
The method targets planar elliptic problems with moving or only partially known
implicit geometry, where repeated solves on a fixed grid replace remeshing. It
is additionally attractive when many queries share one static geometry and an
elliptic field or gradient is required. \edit{The present experiments prescribe
the geometry; they do not evolve a free boundary from the computed PDE.}

The principal structural benefit is that the Cartesian grid, Green's function
table, and sine-transform reconstruction operator remain fixed as the geometry moves.
Body-fitted alternatives may avoid full remeshing through mesh deformation for
moderate motion, but topology changes and appearing obstacles remain difficult.
\edit{The implementation boundary and its timing consequence are documented in
\Cref{sec:method-dynamic,sec:profiling}; the swept-band estimate in
\Cref{sec:analysis-resolve} concerns geometry processing, whereas the dense block
count concerns algebraic assembly.}

\subsection{Limitations}
\label{sec:limitations}
\begin{itemize}
\item \textbf{Narrow passages.} Harmonic gradients can become weak in long
regions feeding narrow channels, and any passage represented by only a few grid
cells is sensitive to both field resolution and path-integration parameters.
\item \textbf{Solve cost.} \edit{In the cached block-update implementation,
full-grid geometry classification, intersection finding, and collocation setup
consume 93--95\% of the per-frame time. The block refresh itself takes only
\SI{1.30}{\milli\second} at $\ell=7$ and
\SI{4.68}{\milli\second} at $\ell=8$, while the cached static block is not
refactored. The reported timings reflect a pure-Python (NumPy/SciPy) prototype;
a swept-band or compiled geometry update is the immediate optimization target.
Larger three-dimensional boundary sets would additionally require compressed
kernel interactions.}
\item \edit{\textbf{Density-system conditioning.} The block update is
accurate through $\ell=8$, but $\kappa_2(B)$ and
$\kappa_2(B_{\mathrm{ss}})$ reach $2.14\times10^4$ and
$1.95\times10^4$, respectively. Substantially finer grids will require a
well-conditioned block formulation or iterative preconditioning; the trace
matrix remains the better-conditioned full-solve baseline.}
\item \edit{\textbf{Discrete trajectory guarantee.} The nodal maximum principle
does not extend automatically to the off-grid interpolated gradient or the
finite-step path integrator. The dynamic scripts implement no segment-level
collision detection or step rejection.}
\item \textbf{Dimensionality.}   The present work is restricted to planar
domains. The extension to 3D is conceptually direct, but the surface solve
becomes the bottleneck; the well-conditioned trace form and kernel compression
discussed in \Cref{sec:factor-compare} are then essential.
\item   \textbf{Boundary conditions.} The current formulation uses
Dirichlet data exclusively, as this suffices for harmonic navigation.
Neumann or Robin data can in principle be incorporated by replacing the
Dirichlet collocation closure with normal-derivative or mixed boundary
functionals, while preserving the fixed-grid and static/dynamic organization.
The resulting solvability and conditioning, however, require separate analysis.
\end{itemize}


\section{Conclusion}
\label{sec:conclusion}

We have developed an unfitted lattice Green's function/difference-potentials
method for repeated Laplace solves on prescribed moving implicit planar domains. The method
uses a fixed Cartesian grid, enforces Dirichlet data at true boundary
intersections through local quadratic basis functions, and reconstructs the interior field
by a sine-transform difference-potential solve. Geometry-independent data are reused
as the boundary moves.

\edit{The directly assembled density matrix admits a
static--dynamic block partition, and the implementation caches and factors its
invariant $B_{\mathrm{ss}}$ once. Subsequent frames assemble only
$B_{\mathrm{sd}}$, $B_{\mathrm{ds}}$, and $B_{\mathrm{dd}}$, solve the reduced
system, and reconstruct the layer trace. On $127^2$ and $255^2$ grids the
result agrees with Algorithm~\ref{alg:full-trace} within $1.1\times10^{-9}$ and reduces
steady-frame time by 2.4--5.5\% relative to that reference. A merge--split--merge test
with 13--23\% dynamic unknowns retains a single static factorization, agrees
within $1.03\times10^{-9}$, and reduces time by 2.5\% relative to
Algorithm~\ref{alg:full-trace}. The full-grid geometry update remains the dominant cost.}

\edit{Manufactured solutions show near-second-order convergence for both the
potential and reconstructed gradient over the full interior and in the first
two interior grid layers.} Tests on a cross-shaped obstacle show that
the method remains stable in the presence of re-entrant corners, although the
available refinements are insufficient to verify the theoretical corner
singularity rate. Dynamic harmonic navigation demonstrates the fixed-grid solver
for translating, deforming, appearing, and topology-changing obstacles.

\edit{The most immediate next step is a swept-band geometry update that avoids
reclassifying the full grid at every frame, together with a sparse full-domain
baseline. Extensions to three dimensions will additionally require compressed
kernel interactions. Although navigation provides the application here, the same
organization applies to other repeated Laplace solves on evolving implicit
boundaries. A particularly natural extension is a quasi-static free-boundary
problem, in which the harmonic field and interface motion are coupled and each
interface update triggers a new Laplace solve.}

\section*{Acknowledgements}

This work was supported in part by the National Natural Science Foundation of
China (Grant No.~12401546) and Wenzhou Kean University (Grant
No.~ISRG2024003).

\small
\setlength{\bibsep}{2pt plus 0.3ex}
\bibliographystyle{plainnat}
\bibliography{references}

\end{document}